\documentclass[11pt]{article}

\usepackage[utf8]{inputenc}
\usepackage[english]{babel}
\usepackage{amsmath}
\usepackage{nicefrac}
\usepackage{amsthm}
\usepackage{amsfonts}
\usepackage{amssymb}
\usepackage{verbatim}
\usepackage{color}
\usepackage[arrow, matrix, curve]{xy}
\usepackage{bbm}
\usepackage{eqparbox}
\usepackage[numbers]{natbib}
\usepackage{stmaryrd}
\usepackage{amssymb}
\usepackage{mathrsfs}
\usepackage{pictexwd,dcpic}
\usepackage{mathabx}

\DeclareMathSymbol{\leq}{\mathrel}{symbols}{20}
   
\DeclareMathSymbol{\geq}{\mathrel}{symbols}{21}

\newtheoremstyle{WreschTheoremstyle} 
                        {1.5em}    
                        {2.5em}    
                        {}         
                        {}         
                        {\bfseries}
                        {}        
                        {\newline} 
                        {\raisebox{0.6em}{\thmname{#1}\thmnumber{#2}\thmnote{ (#3)}}}

\newcommand{\R}{\mathbb{R}}

\newcommand{\N}{\mathbb{N}}
\newcommand{\Z}{\mathbb{Z}}

\newcommand{\E}{\mathbb{E}}
\newcommand{\F}{\mathcal{F}}
\renewcommand{\Pr}{\mathbb{P}}

\renewcommand{\1}{\mathbbm{1}}

\newtheorem{Theorem}{Theorem}[section]
\newtheorem{Proposition}[Theorem]{Proposition}
\newtheorem{Corollary}[Theorem]{Corollary}
\newtheorem{Lemma}[Theorem]{Lemma}
\newtheorem{Remark}[Theorem]{Remark}
\newtheorem{Definition}[Theorem]{Definition}

\numberwithin{equation}{section}

\makeatletter
\newcommand{\customlabel}[1]{%
     \stepcounter{ref}%
   \protected@write
\@auxout{}{\string\newlabel{#1}{{\thesatz.\arabic{ref}}{\thepage}{\thesatz.\arabic{ref}}{#1}{}}}%
   \hypertarget{#1}{\thesatz.\arabic{ref}}%
}
\makeatother

\newenvironment{sciabstract}{\begin{quote}}{\end{quote}}

\newcounter{lastnote}

\title{On Stochastic Cucker-Smale flocking dynamics}
\newcommand{\pdftitle} {On Stochastic Cucker-Smale flocking dynamics}
\newcommand{\pdfauthor}{Martin Friesen}

\author{
Martin Friesen\footnote{Department of Mathematics, Wuppertal University, Germany, friesen@math.uni-wuppertal.de}\\
Oleksandr Kutoviy\footnote{Department of Mathematics, Bielefeld University, Germany, kutoviy@math.uni-bielefeld.de}
}

\usepackage[plainpages=false,pdfpagelabels=true,bookmarks=true,pdfauthor={\pdfauthor},
pdftitle={\pdftitle}]{hyperref}

\makeatletter
\def\HyPsd@CatcodeWarning#1{}
\makeatother

\begin{document}

\maketitle

\begin{sciabstract}\textbf{Abstract:}
We present a stochastic version of the Cucker-Smale flocking dynamics based on a markovian $N$-particle system
of pair interactions with unbounded and, in general, non-Lipschitz continuous interaction potential.
We establish the infinite particle limit $N \to \infty$ and identify the limit as a solution with a nonlinear martingale problem
describing the law of a weak solution to a Vlasov-McKean stochastic equation with jumps. 
Moreover, we estimate the total variation and Wasserstein distance for the time-marginals
from which uniqueness in the class of solutions having some finite exponential moments is deduced.
Based on the uniqueness for the time-marginals we prove uniqueness in law for the Vlasov-McKean equation, 
i.e. we establish propagation of chaos.
\end{sciabstract}


\section{Introduction}

\subsection{Cucker-Smale flocking dynamics}
Cucker and Smale postulated in \cite{CS07a, CS07b} a model for the flocking of birds where convergence to a certain 
consensus (here same direction and velocity in the motion of birds) was shown to depend on the spatial decay of the communication rate between the birds.
In abstract mathematical notation, the Cucker-Smale model describes dynamics of $N$ particles $(r_k,v_k) \in \R^{2d}$,
where $r_k$ stands for the position and $v_k$ for the velocity of the particle with number $k = 1,\dots, N$.
The time evolution is described by the system of ordinary differential equations 
\begin{align}\label{ODE}
 \begin{cases} \frac{dr_k}{dt} &= v_k,
 \\ \frac{dv_k}{dt} &= \frac{1}{N}\sum \limits_{j=1}^{N}\psi(r_k - r_j)(v_j - v_k) \end{cases}.
\end{align}
Here $\psi \geq 0$ is a symmetric function and describes the communication rate between the particles.
Two common examples are
\[
 \psi(r) = \frac{a}{|r|^b} \ \ \text{ or } \ \ \psi(r) = \frac{a}{(1 + |r|^2)^{\frac{b}{2}}}, \ \ \ a,b > 0.
\]
The particular form of \eqref{ODE} implies that the mean velocity is conserved, i.e.
\[
 v_c := \frac{1}{N}\sum \limits_{k=1}^{N}v_k(t) = \frac{1}{N}\sum \limits_{k=1}^{N}v_k(0), \ \ \forall t \geq 0.
\]
Based on Lyapunov functional techniques corresponding to certain dissipative differential inequalities,
the time-asymptotic flocking property
\begin{align*}
 \lim \limits_{t \to \infty}\ \sum \limits_{k=1}^{N}| v_k(t) - v_c|^2 = 0 \ \  \text{ and } \ \ 
 \sup \limits_{t \geq 0}\sum\limits_{k=1}^{N}|r_k(t) - r_c(t)|^2 < \infty
 \end{align*}
was studied in \cite{HL09}, where $r_c(t) := \frac{1}{N}\sum_{k=1}^{N}r_k(t) = r_c(0) + t v_c$ denotes the center of mass.

In many cases one seeks to study properties of the particle dynamics in terms of their associated mean-field equations.
For the classical Cucker-Smale dynamics the corresponding mean-field equation was derived from the BBGKY-hierarchy
when taking the infinite particle limit $N \to \infty$ in \cite{HT08}.
It was shown that the resulting particle density $\mu_t(dr,dv)$ solves the kinetic equation (in the weak formulation)
\begin{align}\label{KINETIC:CS}
 \frac{d}{dt}\int \limits_{\R^{2d}}g(r,v,t)\mu_t(dr, dv) = \int \limits_{\R^{2d}}\left(\frac{\partial g(r,v,t)}{\partial t} + B(\mu_t)g(r,v,t)\right)\mu_t(dr,dv)
\end{align}
where $g$ is a compactly supported, continuously differentiable function and
\begin{align}\label{NONLINEARITY}
 B(\mu_t)g(r,v,t) = v \cdot (\nabla_r g)(r,v,t) - (\nabla_v g)(r,v,t) \cdot \int \limits_{\R^{2d}}\psi(r-q)(v-w)\mu_t(dq,dw).
\end{align}
Existence and uniqueness for measure solutions to \eqref{KINETIC:CS} was established in the class of states where $\mu_t$
has compact support for each $t$ (see \cite{HL09}). For different aspects of this model we refer to \cite{S08, HL09a},
while other related models are studied in \cite{AH10}, \cite{PRT15}, \cite{HJNXZ17}, \cite{CHZ18}.

\subsection{Stochastic Cucker-Smale flocking dynamics}
In this work we propose a stochastic version of the Cucker-Smale model where, roughly speaking, $B(\mu)$ in \eqref{NONLINEARITY}
is replaced by a pure jump operator of mean-field type in the velocity component.
Let $N \geq 2$ be the number of interacting particles $x_k := (r_k,v_k) \in \R^{2d}$, $k = 1, \dots, N$.
Each particle, say $(r_k,v_k)$, may interact with another particle, say $(r_j,v_j)$, and the interaction results in a transition of velocities 
\begin{align}\label{EQ:39}
 v_k \longmapsto v_k + (v_j - v_k + u) = v_j + u,
\end{align}
where $u \in \R^d$ is distributed according to a symmetric probability distribution $a(u)du$.
The rate of this event is supposed to be proportional to $\psi(r_k - r_j)\sigma(v_k - v_j)$, where $\psi, \sigma \geq 0$ are
symmetric functions on $\R^d$.

More precisely, consider a Markov process with phase space $\R^{2dN}$ given, for $F \in C_c^1(\R^{2dN})$, by the Markov generator
\begin{align*}
 &\ (LF)(x_1,\dots, x_N) = \sum \limits_{k = 1}^{N}v_k \cdot (\nabla_{r_k} F)(x_1,\dots, x_N) 
 \\ &\ \ \ + \frac{1}{N}\sum \limits_{k,j=1}^{N}\psi(r_k - r_j)\sigma(v_k-v_j)\int \limits_{\R^d}(F(x_1,\dots, (r_k, v_j + u), \dots, x_N) - F(x_1,\dots, x_N)a(u)du.
\end{align*}
The following are our minimal conditions assumed throughout this work:
\begin{enumerate}
 \item[(A)] $\psi \geq 0$ is continuous, bounded and symmetric.
 \item[(B)] $\sigma \geq 0$ is continuous, symmetric and there exist constants $c_{\sigma} > 0$ and $\gamma \in [0,2]$ such that
 \[
  \sigma(u) \leq c_{\sigma}(1 + |u|^2)^{\gamma / 2}, \ \ u \in \R^d.
 \]
 \item[(C)] $a \geq 0$ is a symmetric probability density on $\R^d$.
\end{enumerate}
For most of the results we also assume that $a$ has some finite moments, i.e. 
\begin{align}\label{EQ:18}
 \lambda_{2p} := \int \limits_{\R^{d}}| u |^{2p}a(u)du < \infty
\end{align}
holds for some $p \geq 0$. The precise value of $p$ will be specified in the corresponding statements.

In Section 3 we will prove that the corresponding martingale problem for the generator $L$ with domain $C_c^1(\R^{2dN})$ is well-posed 
(see Theorem \ref{TH:00}). Moreover, we provide estimates on the moments of this process with constants independent of $N$.

\subsection{The mean-field stochastic Cucker-Smale dynamics}
For each $N \geq 2$, let $(R_k^N, V_k^N)_{k = 1,\dots, N}$ be the Markov process with phase space $\R^{2dN}$ and generator $L$.
In this work we study the infinite particle limit $N \to \infty$ for the sequence of empirical measures
\[
 \mu^{(N)} = \frac{1}{N}\sum \limits_{k=1}^{N}\delta_{(R_k^N, V_k^N)}.
\] 
Denote by $\mathcal{P}(\R^{2d})$ the space of probability measures over $\R^{2d}$.
We prove in Section 3 that each limit of $\mu^{(N)}$ solves a the nonlinear martingale problem with Markov generator
\begin{align}\label{EQ:09}
 (A(\nu)g)(r,v) = v \cdot (\nabla_r g)(r,v) + \int \limits_{\R^{3d}} \left( g(r, w + u) - g(r,v) \right) \psi(r-q)\sigma(v-w)\nu(dq,dw)a(u)du,
\end{align}
in the following sense:
\begin{Definition}
 Let $\mu_0 \in \mathcal{P}(\R^{2d})$. A solution to the nonlinear martingale problem $(A,C_c^1(\R^{2d}), \mu_0)$
 is a probability measure $\mu$ on the Skorokhod space $D(\R_+; \R^{2d})$ such that the following conditions are satisfied
 \begin{enumerate}
  \item[(i)] $\mu( x(0) \in \cdot ) = \mu_0$.
  \item[(ii)] It holds that
  \begin{align}\label{EQ:34}
   \sup \limits_{s \in [0,t]}\E_{\mu}(|v(s)|^{\gamma}) < \infty, \ \ \forall t > 0.
  \end{align}
 where $\E_{\mu}$ denotes the expectation with respect to $\mu$ and $(r,v)$ is the canonical coordinate process on the Skorokhod space.
 \item[(iii)] For each $g \in C_c^1(\R^{2d})$,
  \begin{align}\label{EQ:36}
   g(r(t),v(t)) - g(r(0),v(0)) - \int \limits_{0}^{t}(A(\mu_s)g)(r(s),v(s))ds, \ \ t \geq 0,
  \end{align}
 is a martingale with respect to $\mu$, where $\mu_s$ denotes the time-marginal of $\mu$.
 \end{enumerate}
\end{Definition}
It is possible to write the law $\mu$ also as a weak solution to a certain Vlasov-McKean stochastic equation 
(below always called mean-field SDE) specified in the following definition.
\begin{Definition}
 A process $(R,V)$ is a weak solution to the mean-field SDE, if there exists
 \begin{enumerate}
  \item[(i)] A stochastic basis $(\Omega, \F, (\F_t)_{t \geq 0}, \Pr)$ with the usual conditions.
  \item[(ii)] An $(\F_t)_{t \geq 0}$-adapted Poisson random measure $\mathcal{N}$ on $\R_+ \times [0,1] \times \R^d \times \R_+$ with compensator
  \[
   \widehat{\mathcal{N}}(ds,d\eta,dv,dz) = ds d\eta a(u)du dz
  \]
  defined on $(\Omega, \F, (\F_t)_{t \geq 0}, \Pr)$.
 \item[(iii)] A measurable process $(r_t(\eta),w_t(\eta))$ defined on $([0,1], \mathcal{B}([0,1]), d\eta)$ and
 an $(\F_t)_{t \geq 0}$-adapted, cadlag process $(R,V)$ such that $(r_t,w_t)$ has the same law on $([0,1], \mathcal{B}([0,1]), d\eta)$
 as $(R(t),V(t))$ on $(\Omega, \F, \Pr)$, for each $t \geq 0$.
 \item[(iv)] The process $(R,V)$ satisfies the stochastic equation
 \begin{align}\label{EQ:07}
  \begin{cases} R(t) &= R(0) + \int \limits_{0}^{t}V(s)ds,
   \\ V(t) &= V(0) + \int \limits_{0}^{t}\int \limits_{[0,1] \times \R^d \times \R_+} \widehat{\alpha}(V(s-),R(s),r_s(\eta),w_s(\eta),v,z) \mathcal{N} (ds, d\eta,du,dz)\end{cases},
  \end{align}
  where $(R(0),V(0))$ has law $\mu_0$ and
 \[
  \widehat{\alpha}(V(s-),R(s),r_s(\eta),w_{s}(\eta),v,z) = \left( v + w_{s}(\eta) - V(s-)\right) \1_{[0, \psi(R(s) - r_s(\eta))\sigma(V(s-) - w_{s}(\eta))]}(z).
 \]
 \end{enumerate}
\end{Definition}
The next lemma shows that each solution $\mu$ to the nonlinear martingale problem $(A,C_c^1(\R^{2d}), \mu_0)$
can be represented as a weak solution to the mean-field SDE \eqref{EQ:07}.
\begin{Lemma}\label{LEMMA:03}
 The following assertions hold.
 \begin{enumerate}
  \item[(a)]  Let $(R,V)$ be a weak solution to the mean-field SDE \eqref{EQ:07} satisfying
 \begin{align}\label{EQ:41}
  \sup \limits_{s \in [0,t]} \E(|V(s)|^{\gamma}) < \infty, \ \ \forall t > 0.
 \end{align}
 Then the law of $(R,V)$ on the Skorokhod space $D(\R_+;\R^{2d})$ solves the nonlinear martingale problem $(A, C_c^1(\R^{2d}), \mu_0)$.
 \item[(b)] Let $\mu$ be a solution to the nonlinear martingale problem $(A, C_c^1(\R^{2d}), \mu_0)$.
 Then there exists a weak solution $(R,V)$ to the mean-field SDE \eqref{EQ:07} such that $(R,V)$ has law $\mu$.
 \end{enumerate}
\end{Lemma}
A proof of this Lemma is given in the appendix.
Set $\langle u \rangle := (1 + |u|^2)^{\frac{1}{2}}$, $u \in \R^d$. 
This function satisfies the elementary inequalities
\begin{align}\label{EQ:16}
 \langle u + w \rangle \leq \sqrt{2}\min\{ \langle u \rangle + \langle w \rangle, \langle u \rangle \langle w \rangle \}.
\end{align}
The main result of Section 3 is summarized in the following existence result for the mean-field model.
\begin{Theorem}\label{TH:01}
 Suppose that \eqref{EQ:18} holds for some $2p \geq \max\{4 , 1 + 2\gamma\}$ and let $\mu_0 \in \mathcal{P}(\R^{2d})$ satisfy
 \[
  \int \limits_{\R^{2d}} \left( |r| + |v|^{2p} \right)\mu_0(dr,dv) < \infty.
 \]
 Then there exists a weak solution $(R,V)$ to the mean-field SDE \eqref{EQ:07}.
 Moreover, there exists a constant $C = C(\psi,\sigma,a,p) > 0$ such that
 \begin{align}\label{EQ:26}
  \E\left( \langle V(t)\rangle^{2p} \right) \leq \begin{cases} C \E(\langle V(0)\rangle^{2p}) + C t^{\frac{2p}{2-\gamma}}, & \gamma \neq 2
   \\ \\  \E(\langle V(0)\rangle^{2p})e^{C t}, & \gamma = 2 \end{cases}, \ \ t \geq 0,
 \end{align}
 and, there exists another constant $C' = C'(\psi,\sigma,a) > 0$ such that, for $\gamma \in [0,2]$ and $t \geq 0$,
 \begin{align}\label{EQ:28}
  \E\left( \sup \limits_{s \in [0,t]} \langle V(t)\rangle^{2p-\gamma}\right) \leq \E\left( \langle V(0)\rangle^{2p-\gamma}\right) + C' 2^{2p} \int \limits_{0}^{t}\E\left( \langle V(s)\rangle^{2p}\right)ds.
 \end{align}
\end{Theorem}
In \cite{FRS18a} we have recently studied the mean-field limit for the Enskog process describing the time-evolution of a gas in the vacuum.
The operator $A(\nu)$ defined in \eqref{EQ:09} is less singular then its analogue considered in \cite{FRS18a}.
However, we have not been able to prove that $A(\nu)$ maps compactly supported functions onto bounded functions (unless $\gamma = 0$).
Hence in order to indentify the limits of the empirical measures with solutions to a nonlinear martingale problem additional approximation arguments are required.

The following remark shows that the stochastic Cucker-Smale model still satisfies conservation of momentum.
\begin{Remark}
 Using the particular form of the operator $A(\nu)$ in \eqref{EQ:09} and the symmetry of $\sigma$, $a$, 
 it is not difficult to see that $(R,V)$ satisfies $\sup_{t \in [0,T]} \E(|R(t)|) < \infty$ for all $T > 0$, and
 \[
  \E\left( V(t) \right) = \E\left( V(0) \right), \qquad \E\left( R(t) \right) = \E\left( R(0) \right) + t \E\left( V(0) \right).
 \]
 A similar statement holds also for the particle dynamics.
\end{Remark}
Sections 4 and 5 are devoted to the study of uniqueness for the mean-field model (uniqueness for the nonlinear martingale problem
and uniqueness in law for the mean-field SDE \eqref{EQ:07}. 
Below we formulate only a particular case where $\sigma$ is bounded from which we are able to deduce propagation of chaos,
i.e. convergence of the empirical distributions $\mu^{(N)}$ of the particle dynamics.
\begin{Theorem}\label{MEANFIELD:00}
 Suppose that $\gamma = 0$ and \eqref{EQ:18} holds for $2p = 4$. Let $\mu_0 \in \mathcal{P}(\R^{2d})$ satisfy
 \[
  \int \limits_{\R^{2d}} \left( |r| + |v|^{4} \right) \mu_0(dr,dv) < \infty.
 \]
 Then there exists a unique weak solution $(R,V)$ to the mean-field SDE \eqref{EQ:07}.
 Let $\mu$ be the law of $(R,V)$. Then
 \begin{align}\label{MEANFIELD}
  \frac{1}{N} \sum \limits_{j=1}^{N} \delta_{(R_k^N, V_k^N)} \longrightarrow \mu, \ \ N \to \infty
 \end{align}
 in law on the space of probability measures over the Skorokhod space $D(\R_+;\R^{2d})$.
\end{Theorem}
 Convergence \eqref{MEANFIELD} is a consequence of the uniqueness in law for the mean-field SDE \eqref{EQ:07} and the considerations of Section 3.
 This convergence is also equivalent to the propagation of chaos (see \cite{S91}).
\begin{Remark}
 The moment condition $\int_{\R^{2d}}|v|^4 \mu_0(dr,dv) < \infty$ is to strong and can be replaced by $\int_{\R^{2d}}|v|^2 \mu_0(dr,dv) < \infty$.
 Indeed, if $\gamma = 0$, then we may easily show that the particle dynamics studied in Section 2 preserves second moments with a constant independent of $N$. 
 Moreover, the proofs given in Sections 3 and 4 remain valid in this case, which implies the assertion of Theorem \ref{MEANFIELD:00}.
\end{Remark}
In the particular case where $\sigma$ is bounded, we may also prove that the unique solution propagates exponential moments.
\begin{Corollary}\label{CORR:01}
 Suppose that $\gamma = 0$ and there exist $\delta > 0$ and $\kappa \in (0,1]$ such that
 \begin{align}\label{EQ:38}
  c(\delta,\kappa) = \int \limits_{\R^{d}}e^{\delta|u|^{\kappa}}a(u)du < \infty.
 \end{align}
 is satisfied. Let $\mu_0 \in \mathcal{P}(\R^{2d})$ be such that
 \[
  \int \limits_{\R^{2d}}\left( |r| + e^{\delta |v|^{\kappa}}\right)\mu_0(dr,dv) < \infty.
 \]
 Then there exists a unique weak solution $(R,V)$ to the mean-field SDE \eqref{EQ:07}, and this solution satisfies
 \[
  \E \left( \sup \limits_{s \in [0,t]} e^{\delta |V(s)|^{\kappa}} \right) \leq \int \limits_{\R^{2d}} e^{\delta |v|^{\kappa}}d\mu_0(r,v) e^{Ct}, \ \ t \geq 0
 \]
 for some constant $C > 0$.
\end{Corollary}

\subsection{Structure of the work}
This work is organized as follows.
In Section 2 we first prove some Lyapunov estimates for the particle dynamics.
Then we construct the corresponding Markov process for the particle dynamics and give provide useful moment estimates.
Section 3 is devoted to the infinite particle limit $N \to \infty$ where Theorem \ref{TH:01} is proved.
Uniqueness for the case $\gamma = 0$ is studied in Section 4 from which we deduce Theorem \ref{MEANFIELD:00}
Some further uniqueness results applicable also for the case $\gamma \in (0,2]$ are studied in Section 5, 
i.e. we prove estimates on the total variation and Wasserstein distance for the time-marginals of solutions to the nonlinear martingale problem 
$(A,C_c^1(\R^{2d}), \mu_0)$. The proof of Lemma \ref{LEMMA:03}, some nonlinear generalization of the Gronwall lemma
and a localization argument for martingale problems with unbounded generators are discussed in the appendix.

\section{The particle dynamics}

\subsection{Lyapunov estimates for the particle dynamics}
Let $N \geq 2$ be fixed. The following is one of our main estimates for the moments of the particle system.
\begin{Lemma}\label{LEMMA:00}
 Suppose that \eqref{EQ:18} holds for some $p \geq 2$. 
  Then there exists a constant $C = (\psi,\sigma) > 0$ such that
  \begin{align*}
   \frac{1}{N^2}\sum \limits_{k,j=1}^{N}\psi(r_k-r_j)\sigma(v_k - v_j)\int \limits_{\R^{d}}\left( | v_j + u |^{2p} - | v_k|^{2p}\right)a(u)du
   &\leq  \lambda_{2p} 2^{3p}\frac{C}{N}\sum \limits_{j=1}^{N} \langle v_j \rangle^{2p-2 + \gamma}.
  \end{align*}
\end{Lemma}
Since the proof is elementary and not very interesting we postpone it to the appendix.
Another useful moment estimate is given in the next lemma.
\begin{Lemma}\label{LEMMA:01}
 Suppose that \eqref{EQ:18} holds for $p \geq \frac{1}{2}$.
 Then there exists a constant $C = C(\psi,\sigma) > 0$ such that
 \[
  \frac{1}{N^2}\sum \limits_{k,j=1}^{N}\psi(r_k-r_j)\sigma(v_k - v_j)\int \limits_{\R^{d}}\left| \langle v_j + u\rangle^{2p} - \langle v_k\rangle^{2p}\right|a(u)du
  \leq \lambda_{2p}2^{2p}\frac{C}{N}\sum \limits_{j=1}^{N}\langle v_j\rangle^{2p+\gamma}.
 \]
\end{Lemma}
\begin{proof}
 By the mean-value theorem and \eqref{EQ:16} we find
 \[
  \left| \langle v_j + u\rangle^{2p} - \langle v_k\rangle^{2p}\right| \leq C 2^{2p} \langle u \rangle^{2p} \left( \langle v_j \rangle^{2p} + \langle v_k \rangle^{2p}\right).
 \]
 Hence we obtain
 \begin{align*}
  &\ \frac{1}{N^2}\sum \limits_{k,j=1}^{N}\psi(r_k-r_j)\sigma(v_k - v_j)\int \limits_{\R^{d}}\left| \langle v_j + u\rangle^{2p} - \langle v_k\rangle^{2p}\right|a(u)du
  \\ &\leq \lambda_{2p}2^{2p}\frac{C}{N^2}\sum \limits_{k,j = 1}^{N}\left( \langle v_k \rangle^{\gamma} + \langle v_j \rangle^{\gamma}\right)\left( \langle v_j \rangle^{2p} + \langle v_k \rangle^{2p}\right)
  \leq \lambda_{2p}2^{2p}\frac{C}{N}\sum \limits_{j=1}^{N}\langle v_j \rangle^{2p+\gamma},
 \end{align*}
 where we have used the Young inequality
 \begin{align}\label{EQ:27}
  \langle v_j \rangle^{2p} \langle v_k \rangle^{\gamma} \leq \frac{2p}{2p+\gamma}\langle v_j \rangle^{2p+\gamma} + \frac{\gamma}{2p+\gamma}\langle v_k \rangle^{2p+\gamma}.
 \end{align}
\end{proof}
Finally we give an estimate on the exponential moments.
\begin{Lemma}
 Assume that $\gamma = 0$ and suppose that there exist $\delta > 0$ and $\kappa \in (0,1]$ satisfying \eqref{EQ:38}.
 Then
 \begin{align*}
   \frac{1}{N^2}\sum \limits_{k,j=1}^{N}\sigma(v_k - v_j)\int \limits_{\R^d}\left| e^{\delta \langle v_j + u\rangle^{\kappa} } - e^{\delta \langle v_k \rangle^{\kappa}}\right|a(u)du
\leq \| \sigma\|_{\infty} \frac{1 + e^{\delta}c(\delta,\kappa)}{N} \sum \limits_{j=1}^{N} e^{\delta \langle v_j \rangle^{\kappa}}.
 \end{align*}
\end{Lemma}
\begin{proof}
 Using the inequality $\langle v_j + u\rangle \leq 1 + |v_j| + |u|$ we obtain
 \begin{align*}
  &\ \left| e^{\delta \langle v_j + u\rangle^{\kappa} } - e^{\delta \langle v_k \rangle^{\kappa}}\right|
   \leq e^{\delta \langle v_j + u \rangle^{\kappa}} + e^{\delta \langle v_k \rangle^{\kappa}}
   \leq e^{\delta} e^{\delta |v_j|^{\kappa}}e^{\delta |u|^{\kappa}} + e^{\delta \langle v_k \rangle^{\kappa}}
 \end{align*}
 and hence
 \begin{align*}
  &\ \frac{1}{N^2}\sum \limits_{k,j=1}^{N}\sigma(v_k - v_j)\int \limits_{\R^d}\left| e^{\delta \langle v_j + u\rangle^{\kappa} } - e^{\delta \langle v_k \rangle^{\kappa}}\right| a(u)du
  \\ &\leq \frac{e^{\delta}c(\delta,\kappa)}{N^2}\sum \limits_{k,j=1}^{N}\sigma(v_k - v_j)e^{\delta | v_j|^{\kappa}}
   + \frac{\| \sigma \|_{\infty} }{N}\sum \limits_{k=1}^{N}e^{\delta \langle v_k \rangle^{\kappa}}
  \leq \| \sigma\|_{\infty}\frac{1 + e^{\delta} c(\delta,\kappa)}{N} \sum \limits_{j=1}^{N} e^{\delta \langle v_j \rangle^{\kappa}}.
 \end{align*}
\end{proof}

\subsection{Well-posedness of the martingale problem}
Fix $N \geq 1$.
It is useful to give a pathwise description of the Markov process associated to $L$ in terms of stochastic differential equations.
Namely take a Poisson random measure $\mathcal{N}$ on $\R_+ \times \{1, \dots, N\}^2 \times \R^d \times \R_+$ with compensator 
\begin{align}\label{EQ:20}
 \widehat{\mathcal{N}}(ds,dl,dl',du,dz) = ds \otimes \left(\frac{1}{N}\sum \limits_{j,k = 1}^{N}\delta_{j}(dl)\otimes \delta_k(dl')\right) \otimes (a(u)du) \otimes dz
\end{align}
defined on a stochastic basis $(\Omega, \F, (\F_t)_{t \geq 0}, \Pr)$ with the usual conditions.
The law of the Markov process associated to $L$ should then provide a weak solution to the system of stochastic equations
\begin{align}\label{EQ:03}
 \begin{cases}R(t) &= R(0) + \int \limits_{0}^{t}V(s)ds,
 \\ V(t) &= V(0) + \int \limits_{0}^{t}\int \limits_{\{1, \dots, N\}^2 \times \R^d \times \R_+}G(R(s),V(s-),u,l,l',z) \mathcal{N}(ds,dl,dl',du,dz)\end{cases},
\end{align}
where $e_l = (0,\dots,0, 1,0,\dots,0) \in \R^{dN}$ with the $1$ placed on the $l$-th place and 
\begin{align}\label{EQ:22}
 G(R,V, u, l,l', z) = e_l(u + V_{l'} - V_l) \1_{[0, \psi(R_{l}(s) - R_{l'}(s))\sigma(V_l(s-)-V_{l'}(s-))]}(z).
\end{align}
Let $\mathcal{P}(\R^{2dN})$ be the space of all probability measures on $\R^{2dN}$.
If $\sigma$ is bounded, then weak existence and uniqueness in law for \eqref{EQ:03} can be shown by classical localization arguments (see e.g. \cite{EK86}).
Below we prove a more general statement including all $\gamma \in [0,2]$. 
Since in such a case $LF$ is not bounded, even if $F \in C_c^1(\R^{2dN})$, the desired result does not immediately follows from the 
 classical theory of martingale problems \cite{EK86}. 
 Some additional approximation arguments, combined with moment estimates, are required, i.e. we apply Theorem \ref{APPENDIX:TH00} from the appendix.
\begin{Theorem}\label{TH:00}
 Suppose that \eqref{EQ:18} holds for $p := 2$.
 Then for each $\rho \in \mathcal{P}(\R^{2dN})$ with
 \begin{align}\label{EQ:17}
  \int \limits_{\R^{2dN}} \sum \limits_{j=1}^{N}| v_j |^{4} d\rho(r,v) < \infty
 \end{align}
 the martingale problem $(L, C_c^1(\R^{2dN}), \rho)$ has a unique solution and this solution can be obtained from a weak solution to \eqref{EQ:03}.
\end{Theorem}
\begin{proof}
 Let $g \in C^{\infty}(\R_+)$ be such that $\1_{[0,1]} \leq g \leq \1_{[0,2]}$ and set 
 \[
  g_m(v) = g\left( \frac{\sum_{k=1}^{N}|v_k|^2}{m^2}\right), \ \ v = (v_1,\dots, v_N) \in \R^{dN}.
 \]
 Let $L_m$ be the Markov operator given by $L$ with $\sigma(v_k - v_j)$ replaced by $g_m(v)\sigma(v_k - v_j)$.
 Then for each $F \in C_b^1(\R^{2dN})$ we can find a constant $C = C(F,\psi,\sigma) > 0$ (independent of $m$) such that
 \begin{align}\label{EQ:29}
  |L_mF(x_1,\dots,x_n)|, \ |LF(x_1,\dots,x_n)| \leq C \sum \limits_{j=1}^{N}\langle v_j \rangle^{\gamma}.
 \end{align}
 
\textit{Step 1.} Let $(\Omega, \F, \F_t, \Pr)$ be a stochastic basis and let $(R(0),V(0)) \in \R^{2dN}$ 
 be a random variable with some given law $\mu \in \mathcal{P}(\R^{2dN})$.
 Let $\mathcal{N}_m$ be a Poisson random measure on $\Omega$ with compensator 
 \[
  \widehat{\mathcal{N}}_m(ds,dl,dl',du,dz) = ds \otimes \left(\frac{1}{N}\sum \limits_{j,k = 1}^{N}\delta_{j}(dl)\otimes \delta_k(dl')\right) \otimes (a(u)du) \otimes dz
 \]
 on $\R_+\times \{1,\dots, N\}^2 \times \R^d \times [0, c_m]$ (for some constant $c_m > 0$ large enough). Then 
 \[
  \widehat{\mathcal{N}}_m((0,t] \times \{1,\dots, N\} \times \R^d \times [0,c_m]) < \infty, \ \ \forall t > 0
 \]
 and hence the system of stochastic equations 
 \begin{align}\label{EQ:10}
  \begin{cases}R^m(t) &= R(0) + \int \limits_{0}^{t}V^m(s)ds,
  \\ V^m(t) &= V(0) + \int \limits_{0}^{t}\int \limits_{\{1, \dots, N\}^2 \times \R^d \times [0,c_m]}G^m(R^m(s),V^m(s-),u,l,l',z) \mathcal{N}_m(ds,dl,dl',du,dz)\end{cases},
 \end{align}
 with
 \[
  G^{m}(R^m,V^m, u,l,l',z) = e_l (u + V_{l'}^m - V_l^m)\1_{[0, \psi(R_{l}^m(s) - R_{l'}^m(s))g_m(V^m(s-))\sigma(V_l^m(s-)-V_{l'}^m(s-))]}(z)
 \]
 can be uniquely solved from jump to jump.
 From \cite{KURTZ10} we conclude that the martingale problem $(L_m, C_c^1(\R^{2dN}),\mu)$ has, for each $\mu \in \mathcal{P}(\R^{2dN})$,
 a unique solution whose law can be obtained from \eqref{EQ:10}. 
 
\textit{Step 2.} Suppose that $(R(0),V(0))$ has law $\rho$ satisfying \eqref{EQ:17}.
 In order to apply Theorem \ref{APPENDIX:TH00} it suffices to show that
 \begin{align}\label{EQ:05}
  \sup \limits_{m \geq 1} \Pr( \sup \limits_{s \in [0,t]} \langle V^m(s)\rangle^2 ) + \sup \limits_{m \geq 1}\sup \limits_{s \in [0,t]}\E( \langle V^m(s) \rangle^{4} ) < \infty, \ \ \forall t > 0.
 \end{align}
 Using the It\^{o} formula and Lemma \ref{LEMMA:00} we deduce for some constant $C = C(\psi,\sigma) > 0$ (independent of $m$)
 \[
  \E\left( \frac{1}{N}\sum \limits_{k=1}^{N}\langle V_k^m(t)\rangle^{4} \right) \leq \left( \E\left( \frac{1}{N}\sum \limits_{k=1}^{N}\langle V_k(0)\rangle^{4} \right) \right)e^{C t}, \ \ t \geq 0,
 \]
 and likewise we deduce from Lemma \ref{LEMMA:01} 
 \[
  \E\left( \frac{1}{N}\sup \limits_{s \in [0,t]} \sum \limits_{k=1}^{N}\langle V_k^m(s)\rangle^{4-\gamma} \right) 
  \leq \E\left( \frac{1}{N}\sum \limits_{k=1}^{N}\langle V_k(0) \rangle^{4-\gamma} \right) + C \int\limits_{0}^{t}\E\left( \frac{1}{N}\sum \limits_{k=1}^{N}\langle V_k^m(s)\rangle^{4}\right) ds.
 \] 
 This proves \eqref{EQ:05}.
 Hence we may apply Theorem \ref{APPENDIX:TH00} to conclude that the martingale problem for $(L,C_c^1(\R^{2dN}), \rho)$ has a unique solution
 $\Pr_{\rho}$ which satisfies
 \[
  \sup \limits_{s \in [0,t]} \E_{\rho}\left( \sum \limits_{k=1}^{N}\langle v_k(s) \rangle^{4} \right)ds < \infty, \ \ \forall t > 0,
 \]
 where $\E_{\rho}$ denotes the integration w.r.t. $\Pr_{\rho}$ and $(r(t),v(t))$ the coordinate process in the Skorokhod space $D(\R_+;\R^{2dN})$.
 
 \textit{Step 3.} By construction of $\Pr_{\rho}$ we see that, for any $F \in C_c^1(\R^{2dN})$,
 \begin{align}\label{EQ:06}
  F(r(t),v(t)) - F(r(0), v(0)) - \int \limits_{0}^{t}(LF)(r(s),v(s))ds, \ \ t \geq 0
 \end{align}
 is a martingale with respect to $\Pr_{\rho}$. 
 In view of \eqref{EQ:29} we conclude that \eqref{EQ:06} is a local martingale for any $F \in C_b^1(\R^{2dN})$.
 Existence of a weak solution $(R,V)$ to \eqref{EQ:03} having the prescribed law $\Pr_{\rho}$ can be now obtained from \cite[Theorem A.1]{HK90}.
\end{proof}
\begin{Remark}
 Suppose that $\psi$ and $\sigma$ are locally Lipschitz continuous. Then similar arguments to \cite{G92} can be used to prove 
 strong existence and pathwise uniqueness for \eqref{EQ:03}.
\end{Remark}
We call $\rho \in \mathcal{P}(\R^{2dN})$ symmetric, if for any permutation $\tau$ of $\{1,\dots, N\}$
and any bounded measurable function $F: \R^{2dN} \longrightarrow \R$
\[
 \int \limits_{\R^{2dN}}F(x_1,\dots, x_N) d\rho(x_1,\dots, x_N) = \int \limits_{\R^{2dN}}F(x_{\tau(1)},\dots, x_{\tau(N)}) d\rho(x_1,\dots, x_N).
\]
The following corollary shows that the particles trajectories are indistinguishable.
\begin{Corollary}\label{CORR:05}
 Let $\rho$ be symmetric with property \eqref{EQ:17} and suppose that \eqref{EQ:18} holds for $p = 2$.
 Denote by $X_k^N := (R_k^N,V_k^N)$, $k = 1,\dots, N$, the unique weak solution to \eqref{EQ:03}.
 Then $X_1^N, \dots, X_N^N$ are exchangeable as elements in $D(\R_+;\R^{2d})$, i.e.
 for any permutation $\tau$ of $\{1,\dots, N\}$ and any bounded measurable function $F: D(\R_+;\R^{2dN}) \longrightarrow \R$
 \begin{align}\label{EQ:40}
  \E(F( X_1^N, \dots, X_N^N ) ) = \E(F(X^N_{\tau(1)}, \dots, X^N_{\tau(N)}) ).
 \end{align}
 In particular, $(R_k^N,V_k^N)$, $k = 1,\dots, N$, are identically distributed as elements in $D(\R_+;\R^{2d})$.
\end{Corollary}
\begin{proof}
 Since $L$ maps symmetric functions onto symmetric functions,
 the assertion follows from uniqueness of the martingale problem $(L, C_c^1(\R^{2dN}), \rho)$.
\end{proof}

\subsection{Moments of the particle dynamics}
Fix $N \geq 1$.
Below we prove some moment estimates (uniform in $N$) for the unique solution to \eqref{EQ:03}.
\begin{Corollary}\label{CORR:00}
 Suppose that \eqref{EQ:18} holds for some $p \geq 2$ and let $\rho \in \mathcal{P}(\R^{2dN})$ be symmetric with
 \[
  \int \limits_{\R^{2dN}}\sum \limits_{j=1}^{N}|v_j|^{2p} d\rho(r,v) < \infty.
 \]
 Let $(R_k^N, V_k^N)_{k = 1,\dots, N}$ be the unique solution to \eqref{EQ:03} defined on a stochastic basis $(\Omega^N, \F^N, (\F_t^N)_{t \geq 0},\Pr^N)$.  Then there exists a constant $C = C(\psi,\sigma) > 0$ (independent of $N$) such that, for $\gamma \in [0,2)$,
 \[
   \E^N\left( \frac{1}{N}\sum \limits_{j=1}^{N}\langle V_j^N(t)\rangle^{2p} \right) 
   \leq 2^{\frac{4p^2}{2-\gamma}} \E^N\left( \frac{1}{N} \sum \limits_{j=1}^{N}\langle V^N_j(0)\rangle^{2p}\right) + \left( C_p \frac{2-\gamma}{2p}\right)^{\frac{2p}{2-\gamma}}t^{\frac{2p}{2-\gamma}}, \ \ t \geq 0,
 \]
 where $C_p = C \lambda_{2p}2^{5p}$, and, for $\gamma = 2$,
 \begin{align}\label{EQ:12}
   \E^N\left( \frac{1}{N}\sum \limits_{j=1}^{N}\langle V_j^N(t)\rangle^{2p} \right) \leq 2^{2p}\E\left( \frac{1}{N}\sum \limits_{j=1}^{N}\langle V^N_j(0)\rangle^{2p} \right)e^{C_p t}, \ \ t \geq 0
 \end{align}
 Moreover, there exists another constant $C' = C'(\psi,\sigma) > 0$ such that
 \begin{align}\label{EQ:13}
  \E^N\left( \sup \limits_{t \in [0,T]} \langle V_1^N(t)\rangle^{2p-\gamma} \right) 
  \leq \E^N\left( \frac{1}{N}\sum \limits_{j=1}^{N}\langle V_j^N(0) \rangle^{2p-\gamma} \right) + C' \lambda_{2p} 2^{2p} \int \limits_{0}^{T}\E^N\left( \frac{1}{N}\sum \limits_{j=1}^{N}\langle V^N_j(s)\rangle^{2p}\right) ds
 \end{align}
\end{Corollary}
\begin{proof}
 It follows from Lemma \ref{LEMMA:00} and the It\^{o} formula that
 \begin{align*}
  \E^N\left( \frac{1}{N}\sum \limits_{j=1}^{N}| V_j^N(t)|^{2p} \right) 
  &\leq \E^N\left( \frac{1}{N}\sum \limits_{j=1}^{N}| V_j^N(0)|^{2p} \right) + C \lambda_{2p} 2^{3p} \int \limits_{0}^{t}\E^N\left( \frac{1}{N}\sum \limits_{j=1}^{N} \langle V_j^N(s) \rangle^{2p-2 + \gamma} \right)ds
  \\ &\leq \E^N\left( \frac{1}{N}\sum \limits_{j=1}^{N}| V_j^N(0)|^{2p} \right) + C \lambda_{2p}2^{3p} \int \limits_{0}^{t} \E^N\left( \frac{1}{N}\sum \limits_{j=1}^{N} \langle V_j^N(s) \rangle^{2p} \right)^{1 - \frac{2-\gamma}{2p}} ds
 \end{align*}
 where we have used the Jensen inequality. Next observe that,
 by $1 + |v|^{2p} \leq \langle v \rangle^{2p}$ and previous estimate,
 \begin{align*}
  \E^N\left( \frac{1}{N}\sum \limits_{j=1}^{N}\langle V_j^N(t)\rangle^{2p} \right) 
  &\leq 2^{2p} + 2^{2p}\E^N\left( \frac{1}{N}\sum \limits_{j=1}^{N}| V_j^N(t)|^{2p} \right)
  \\ &\leq 2^{2p} \E^N\left( \frac{1}{N}\sum \limits_{j=1}^{N}\langle V_j^N(0)\rangle^{2p} \right)  + C_p \int \limits_{0}^{t} \E^N\left( \frac{1}{N}\sum \limits_{j=1}^{N} \langle V_j^N(s) \rangle^{2p} \right)^{1 - \frac{2-\gamma}{2p}} ds. 
 \end{align*}
 For $\gamma = 2$ we apply the Gronwall lemma, for $\gamma \in [0,2)$ we may apply a nonlinear version of the Gronwall lemma stated in the appendix.
 Finally assertion \eqref{EQ:13} follows by the It\^{o} formula, similar arguments to Lemma \ref{LEMMA:01} and Corollary \ref{CORR:05}.

To be more rigorous one has to consider the above estimates first for the variables $V^{N,m}(t) := V^N(t \wedge \tau_m)$ 
 where $\tau_m$ is a stopping time choosen in such a way that $V^N(t \wedge \tau_m)$ is bounded.
 Obtaining the desired estimates for $V^{N,m}(t)$ (with constants independent of $m$), one may then pass to the limit $m \to \infty$. 
 Since such type of arguments are rather standard, we leave the details for the reader.
\end{proof}
Using similar arguments and Lemma \ref{LEMMA:01} we can show propagation of exponential moments.
\begin{Corollary}\label{CORR:02}
 Suppose that $\gamma = 0$ and there exist $\delta > 0$ and $\kappa \in (0,1]$ such that \eqref{EQ:38} holds.
  Then there exists a constant $C = C(\psi,\sigma)> 0$ such that
  \[
   \E^N\left( \frac{1}{N} \sum \limits_{j=1}^{N} e^{\delta \langle V_j^N(t) \rangle^{\kappa}} \right) 
   + \E^N\left( \sup \limits_{s \in [0,t]} e^{\delta \langle V_1^N(t) \rangle^{\kappa}} \right)
   \leq \E^N\left( \frac{1}{N} \sum \limits_{j=1}^{N} e^{\delta \langle V_j^N(0) \rangle^{\kappa}} \right)e^{C t}.
  \]
\end{Corollary}

\section{The infinite particle limit $N \to \infty$}
In this section we perform the limit $N \to \infty$ and identify the corresponding limiting process, i.e. we prove Theorem \ref{TH:01}.
Corollary \ref{CORR:01} can be deduced by the same arguments but now using the moment estimates from Corollary \ref{CORR:02}.

For each $N \geq 2$, let $\rho^{(N)} \in \mathcal{P}(\R^{2dN})$ be given by 
\[
  \rho^{(N)}(dr_1,dv_1,\dots, dr_N,dv_N) = \bigotimes_{k = 1}^{N} \mu_0(dr_k,dv_k)
\]
and denote by $(R_k^N,V_k^N)_{k=1,\dots,N}$ the unique weak solution to \eqref{EQ:03} defined on a 
stochastic basis $(\Omega_N, \F^N, (\F_t^N)_{t \geq 0}, \Pr^N)$ with the usual conditions.
Denote by $\mathcal{P}(D(\R_+;\R^{2d}))$ the space of probability measures over the Skorokhod space $D(\R_+;\R^{2d})$ and,
similarly let $\mathcal{P}(\mathcal{P}(D(\R_+;\R^{2d})))$ be the space of probability measures over $\mathcal{P}(D(\R_+;\R^{2d}))$.
Define a sequence of empirical measures
\begin{align}\label{EQ:33}
 \mu^{(N)} = \frac{1}{N} \sum \limits_{k=1}^{N} \delta_{(R_k^N, V_k^N)},
\end{align}
i.e. random variables with values in $\mathcal{P}( D(\R_+;\R^{2d}) )$ and 
denote by $\pi^{(N)} \in \mathcal{P}(\mathcal{P}(D(\R_+;\R^{2d})))$ the law of $\mu^{(N)}$.
The proof consists of the following two steps
\begin{enumerate}
 \item[Step 1.] Prove that $\pi^{(N)}$ is relatively compact and show that each limit is supported on processes having the desired moment bounds.
 \item[Step 2.] Prove that each limit $\pi^{(\infty)}$ of a subsequence of $\pi^{(N)}$ is supported on solutions to the nonlinear martingale problem 
 $(A, C_c^1(\R^{2d}), \mu_0)$.
\end{enumerate}

\subsection{Compactness and moment estimates}
Let us show that $(\pi^{(N)})_{N \geq 2}$ is relatively compact.
\begin{Proposition}
 $(\pi^{(N)})_{N \geq 2}$ is relatively compact in $\mathcal{P}(\mathcal{P}(D(\R_+;\R^{2d})))$.
\end{Proposition}
\begin{proof}
 In view of \cite[Proposition 2.2]{S91}, see also Corollary \ref{CORR:05}, it suffices to show that $(R_1^N,V_1^N)$ is tight in $D(\R_+;\R^{2d})$.
 First we observe that
 \[
  \sup \limits_{t \in [0,T]}\E^N\left(|R_1^N(t)|\right) \leq \sup \limits_{N \geq 2}\E^N\left( | R_1^N(0) |\right) + T \sup \limits_{N \geq 2}\E^N\left( \sup \limits_{t \in [0,T]}| V_1^N(t)|\right) < \infty,
 \]
 where the right-hand side is finite due to the moment estimates of previous section.
 We seek to apply the Aldous criterion (see e.g. \cite{JS03}).
 For each $N \geq 2$ let $S^N,T^N$ be $(\F_t^N)_{t \geq 0}$ stopping times such that for $M \in \N$ and $\delta \in (0,1]$
 we have $S^N \leq T^N \leq S^N + \delta$ and $S^N,T^N \leq M$. Then
 \[
  \E^N\left( |R_1^N(T^N) - R_1^N(S^N)| \right) \leq \delta \sup \limits_{N \geq 2}\E^N\left( \sup \limits_{\tau \in [0,M]}|V_1^N(\tau)|\right)
 \]
 and similarly by \eqref{EQ:03}
 \begin{align*}
  &\ \E^N\left( |V_1^N(T^N) - V_1^N(S^N)| \right) 
  \\ &\leq \frac{C}{N}\sum \limits_{j=1}^{N}\E^N\left( \int \limits_{S^N}^{T^N}\int \limits_{\R^d}| u + (V_j^N(\tau) - V_1^N(\tau))| \left( \langle V_1^N(\tau)\rangle^{\gamma} + \langle V_j^N(\tau)\rangle^{\gamma}\right)a(u)dud\tau \right)
  \\ &\leq C \delta \sup \limits_{N \geq 2} \E^N\left( \sup \limits_{\tau \in [0,M]}\langle V_1^N(\tau)\rangle^{1+\gamma}\right).
 \end{align*}
 Since $2p \geq \max\{4, 1+2\gamma\}$, the moment estimates of previous section imply that the right-hand sides are finite.
 This proves the assertion.
\end{proof}
For $\nu \in \mathcal{P}(D(\R_+;\R^{2d}))$ let $\nu_t \in \mathcal{P}(\R^{2d})$ be the time-marginal at time $t \geq 0$ and, for $q \geq 0$, set
\[
 \| \nu_t \|_q := \int \limits_{\R^{2d}} \langle v \rangle^q \nu_t(dr,dv).
\]
The next lemma provides moment estimates for the limits of the empirical measure.
\begin{Lemma}\label{LEMMA:02}
 There exists a constant $C = C(\psi,\sigma) > 0$ such that for all $t \geq 0$ we have, for $\gamma \in [0,2)$,
 \[
  \int \limits_{\mathcal{P}(D(\R_+;\R^{2d}))}\| \nu_t \|_{2p} d\pi^{(\infty)}(\nu) 
  \leq 2^{\frac{4p^2}{2-\gamma}} \int \limits_{\R^{2d}}\langle v \rangle^{2p} \mu_0(dr,dv) + \left(C_p \frac{2-\gamma}{2p}\right)^{\frac{2p}{2-\gamma}} t^{\frac{2p}{2-\gamma}},
 \]
 with $C_p = C \lambda_{2p} 2^{5p}$ and, for $\gamma = 2$,
 \[
  \int \limits_{\mathcal{P}(D(\R_+;\R^{2d}))}\| \nu_t \|_{2p} d\pi^{(\infty)}(\nu) 
  \leq \left( \int \limits_{\R^{2d}}\langle v \rangle^{2p} \mu_0(dr,dv) \right)e^{C_p t}.
 \]
\end{Lemma}
\begin{proof}
 By approximation and the Lemma of Fatou we get
 \begin{align*}
  \int \limits_{\mathcal{P}(D(\R_+;\R^{2d}))}\| \nu_t \|_{2p} d\pi^{(\infty)}(\nu) 
  &\leq \sup \limits_{N \geq 2} \int \limits_{\mathcal{P}(D(\R_+;\R^{2d}))}\| \nu_t \|_{2p} d\pi^{(N)}(\nu) 
  \\ &= \sup \limits_{N \geq 2} \E^N\left( \frac{1}{N} \sum \limits_{j=1}^{N}\langle V_j^N(t) \rangle^{2p} \right).
 \end{align*}
 The assertion now follows from Corollary \ref{CORR:00}.
\end{proof}
From this we readily deduce, after we have completed Step 2 and Step 3, the desired moment estimates \eqref{EQ:26}. 
Estimate \eqref{EQ:28} follows from the It\^{o} formula and a direct computation.

\subsection{Identifying the limit}
The following shows that each limit point $\pi^{(\infty)}$ of a subsequence of $(\pi^{(N)})_{N \geq 2}$
is supported on solutions to the nonlinear martingale problem $(A,C_c^1(\R^{2d}), \mu_0)$.
\begin{Proposition}\label{PROP:00}
 Let $\pi^{(\infty)} \in \mathcal{P}( \mathcal{P}(D(\R_+;\R^{2d})) )$ be any weak limit of a subsequence of $(\pi^{(N)})_{N \geq 2}$.
 Then $\pi^{(\infty)}$-a.a. $\mu \in \mathcal{P}(D(\R_+;\R^{2d}))$ solve the nonlinear martingale problem $(A, C_c^1(\R^{2d}), \mu_0)$.
\end{Proposition}
The rest of this section is devoted to the proof of this proposition.
 It is not difficult to see that the complement of 
\[
 D_{\mu} = \left\{ t > 0 \ | \ \mu \left(  (r,v) \in D(\R_+;\R^{2d}) : (r(t),v(t)) = (r(t-),v(t-))  \right) = 1  \right\}
\]
 is at most countable
 and the coordinate function $(r,v) \longmapsto (r(t),v(t))$ is $\mu$-a.s. continuous, for any $t \in D_{\mu}$ and any $\mu \in \mathcal{P}(D(\R_+;\R^{2d}))$.
 Moreover, we can show that also the complement of
 \[
  D(\pi^{(\infty)}) = \left\{ t > 0 \ | \ \pi^{(\infty)}\left(  \mu \in \mathcal{P}(D(\R_+;\R^{2d})) \ : \ t \in D_{\mu} \right)  = 1 \right\}
 \]
 is at most countable.

 Let $0 \leq t_1, \dots, t_m \leq s \leq t$ with $t_1,\dots, t_m,s,t \in D(\pi^{(\infty)})$, 
 $m \in \N$, $g_1,\dots, g_m \in C_b(\R^{2d})$ and $g \in C_c^1(\R^{2d})$.
 For $(r,v) \in D(\R_+;\R^{2d})$ and $\mu \in \mathcal{P}(D(\R_+;\R^{2d}))$ set
 \begin{align}\label{CS:00}
  H(\mu;r,v) := \left(g(r(t),v(s)) - g(r(s),v(s)) - \int \limits_{s}^{t}(A(\mu_{\tau})g)(r(\tau),v(\tau))d\tau\right) \prod \limits_{j=1}^{m}g_j(r(t_j),v(t_j))
 \end{align}
 and define
 \begin{align}\label{EQ:24}
  F(\mu) := \int \limits_{D(\R_+;\R^{2d})}H(\mu;r,v)\mu(dr,dv).
 \end{align}
 It is clear that $\mu$ is a solution to the nonlinear martingale problem  $(A,C_c^1(\R^{2d}), \mu_0)$, provided
 $\mu(x(0) \in \cdot) = \mu_0$, \eqref{EQ:34} holds and $F(\mu) = 0$.
 Since, by Lemma \ref{LEMMA:02}, $\pi^{(\infty)}$-a.a. $\mu$ satisfy \eqref{EQ:34} and $\mu(x(0) \in \cdot) = \mu_0$, it suffices to show that 
 \begin{enumerate}
  \item[(a)] $\lim_{N \to \infty} \int_{\mathcal{P}(D(\R_+;\R^{2d}))} |F(\mu)|^2 d\pi^{(N)}(\mu) = 0$,
  \item[(b)] $\lim_{N \to \infty} \int_{\mathcal{P}(D(\R_+;\R^{2d}))} |F(\mu)| d\pi^{(N)}(\mu)
   = \int_{\mathcal{P}(D(\R_+;\R^{2d}))} |F(\mu)| d\pi^{(\infty)}(\mu)$,
 \end{enumerate}
 where for simplicity of notation $\pi^{(N)}$ denotes the subsequence converging weakly to $\pi^{(\infty)}$.
 Let us first prove (a).
\begin{Lemma}
 Assertion (a) is satisfied.
\end{Lemma}
\begin{proof}
 Let $\widetilde{\mathcal{N}}(ds,dl,dl',du,dz)$ be the compensated Poisson random measure
 on $\R_+ \times \{1, \dots, N\}^2 \times \R^{d} \times \R_+$ with compensator given by \eqref{EQ:20}, 
 \begin{align}\label{EQ:100}
  G(r,v, u, l,l', z) = e_l(u + v_{l'} - v_l) \1_{[0, \psi(r_{l} - r_{l'})\sigma(v_l - v_{l'})]}(z),
 \end{align}
 where $(r,v) \in \R^{2dN}$, $z \in \R_+$, $u \in \R^d$ and $(l,l') \in \{1,\dots, N\}^2$ is defined as in \eqref{EQ:22} and set
 \[
 M_{s,t}^{N,k} = \int_{s}^{t}\int_{E}\left( g(R_k^{N}(\tau), V_k^{N}(\tau-) + G_k) -  g(R_k^{N}(\tau),V_k^{N}(\tau-))\right)\widetilde{\mathcal{N}}(d\tau,dl,dl',du,dz),
 \]
 where $E := \{1,\dots,N\}^2\times \R^d \times \R_+$ and $G_k = G_k(R^N(\tau), V^N(\tau-), u,l,l',z)$. Then
 \begin{align*}
  &\ (A(\mu^{(N)})g)(R_k^N,V_k^N) = V_k^{N} \cdot (\nabla_r g)(R_k^N, V_k^N)
  \\ &+ \frac{1}{N}\sum \limits_{j=1}^{N}\psi(R_k^N - R_j^N)\sigma(V_k^N - V_j^N) \int \limits_{\R^d}\left( g(R_k^N, V_j^N + u) - g(R_k^N,V_k^N)\right)a(u)du
 \end{align*}
 and from the It\^{o} formula one immediately obtains
 \begin{align*}
  g(R_k^N(t),V_k^N(t)) = g(R_k^N(s),V_k^N(s)) + \int \limits_{s}^{t}(A(\mu_{\tau}^{(N)})g)(R_k^N(\tau),V_k^N(\tau))d\tau + M_{s,t}^{N,k}.
 \end{align*}
 This shows that
 \[
  F(\mu^{(N)}) = \frac{1}{N} \sum \limits_{k=1}^{N}H(\mu^{(N)}, X_k^N) = \frac{1}{N}\sum_{k=1}^{N}M_{s,t}^{N,k}\prod_{j=1}^{m}g_j(R_k^{N}(t_j), V_k^N(t_j)).
 \]
 For the Doob-Meyer process of $M_{s,t}^{N,k}$ we obtain
 \begin{align*}
  \langle M_{s,t}^{N,k} \rangle &= \frac{1}{N} \sum \limits_{j = 1}^{N} \int_{s}^{t}\int_{\R^d}\left( g(R_k^{N}, V_j^N + u) -  g(R_k^{N},V_k^{N})\right)^2 \psi(R_k^N - R_j^N) \sigma(V_k^N - V_j^N) d\tau a(u)du
  \\ &\leq \frac{C}{N} \sum \limits_{j = 1}^{N} \int \limits_{s}^{t} \left( \langle V_k^N(\tau) \rangle^{\gamma} + \langle V_j^N(\tau) \rangle^{\gamma} \right) d\tau,
 \end{align*}
 which implies, in view of the moment estimates of previous section, $\E^N( \langle M_{s,t}^{N,k} \rangle) \leq C$
 for all $k = 1,\dots, N$ where the constant $C = C(\psi,\sigma,a,g)$ is independent of $N$.
 Using the particular form of $G$ defined in \eqref{EQ:100}, we obtain for the covariation process
 $\langle M_{s,t}^{N,k}, M_{s,t}^{N,j} \rangle = 0$ for all $k \neq j$.
 Hence we conclude from the properties of the processes $\langle M_{s,t}^{N,k} \rangle$ and $\langle M_{s,t}^{N,k}, M_{s,t}^{N,j} \rangle$
 \begin{align*}
  &\ \int \limits_{\mathcal{P}(D(\R_+;\R^{2d}))}|F(\nu)|^2 d\pi^{(N)}(\nu)
  \\ &= \frac{1}{N^2}\sum \limits_{k \neq j} \E^N\left( M_{s,t}^{N,k}M_{s,t}^{N,j}\prod\limits_{l_1 = 1}^{m}g_{l_1}(X_k^N(t_{l_1}))\prod\limits_{l_2 = 1}^{m}g_{l_2}(X_k^N(t_{l_2})) \right)
 \\ &\ \ \  + \frac{1}{N^2}\sum \limits_{k=1}^{N} \E^N \left( ( M_{s,t}^{N,k})^2 \prod\limits_{l = 1}^{m}g_{l}(X_k^N(t_{l}))^2 \right)
 \\ &=  \frac{1}{N^2}\sum \limits_{k \neq j} \E^N\left( \langle M_{s,t}^{N,k}, M_{s,t}^{N,j} \rangle \prod\limits_{l_1 = 1}^{m}g_{l_1}(X_k^N(t_{l_1}))\prod\limits_{l_2 = 1}^{m}g_{l_2}(X_k^N(t_{l_2})) \right)
 \\ &\ \ \  + \frac{1}{N^2}\sum \limits_{k=1}^{N} \E^N \left( \langle M_{s,t}^{N,k} \rangle \prod\limits_{l = 1}^{m}g_{l}(X_k^N(t_{l}))^2 \right)
  \leq \frac{C(\psi,\sigma,a,g,g_1,\dots, g_m)}{N},
 \end{align*}
 which proves the assertion.
\end{proof}
Next we prove that assertion (b) holds.
\begin{Lemma}
 Assertion (b) is satisfied.
\end{Lemma}
\begin{proof}
 Take $\varphi \in C^{\infty}(\R_+)$ with $\1_{[0,1]} \leq \varphi \leq \1_{[0,2]}$.
 For $R > 0$ and $\nu \in \mathcal{P}(\R^{2d})$ let
 \begin{align*}
  (A_R(\nu)g)(r,v) &= v \cdot (\nabla_r g)(r,v) 
  \\ &\ + \int \limits_{\R^{2d} \times \R^d} \varphi\left( \frac{|w|^2}{R^2}\right) \left( g(r, w + u) - g(r,v) \right) \psi(r-q)\sigma(v-w)\nu(dq,dw)a(u)du.
 \end{align*}
 Then it is not difficult to see that 
 \[
  \mathcal{P}(\R^{2d}) \times \R^{2d} \ni (\nu,r,v) \longmapsto (A_R(\nu)g)(r,v)
 \]
 is jointly continuous where $\mathcal{P}(\R^{2d})$ is endowed with the topology of weak convergence. 
 Moreover one can show that for some constant $C = C(\psi, \sigma, a,g)$
 \begin{align}\label{EQ:25}
  |A_R(\nu)g(r,v) - A(\nu)g(r,v)| \leq C \int \limits_{\R^{2d}} \1_{ \{ |w| > R \} } \sigma(v-w)d\nu(q,w) \leq \frac{C}{R^{1/2}}\| \nu \|_{\gamma + \frac{1}{2}} \langle v \rangle^{\gamma}.
 \end{align}
 Let $H_R$ be defined by \eqref{CS:00} with $A$ replaced by $A_R$ and define $F_R(\mu)$ by \eqref{EQ:24} with $H$ replaced by $H_R$. 
 Then we obtain
 \begin{align*}
  &\ \left| \int \limits_{\mathcal{P}(D(\R_+;\R^{2d}))}|F(\mu)| d\pi^{(N)}(\mu) - \int \limits_{\mathcal{P}(D(\R_+;\R^{2d}))}|F(\mu)| d\pi^{(\infty)}(\mu) \right|
  \\ &\leq  \int \limits_{\mathcal{P}(D(\R_+;\R^{2d}))}|F(\mu) - F_R(\mu)| d\pi^{(N)}(\mu) 
 + \int \limits_{\mathcal{P}(D(\R_+;\R^{2d}))}|F_R(\mu) - F(\mu)| d\pi^{(\infty)}(\mu) 
  \\ &\ \ \ + \left| \int \limits_{\mathcal{P}(D(\R_+;\R^{2d}))}|F_R(\mu)| d\pi^{(N)}(\mu) - \int \limits_{\mathcal{P}(D(\R_+;\R^{2d}))}|F_R(\mu)| d\pi^{(\infty)}(\mu) \right|
  \\ &= I_1 + I_2.
 \end{align*}
 Using \eqref{EQ:25} we obtain for $T > t$ and some constant $C = C(g, g_1,\dots, g_m, \psi, \sigma, a)$
 \begin{align*}
  |F(\mu) - F_R(\mu)| &\leq C \int \limits_{s}^{t}\int \limits_{D(\R_+;\R^{2d})} | (A(\mu_{\tau})g)(r(\tau),v(\tau)) - (A_R(\mu_{\tau})g)(r(\tau),v(\tau))| \mu(dr,dv)d\tau
  \\ &\leq \frac{C}{R^{1/2}}\int \limits_{s}^{t} \| \mu_{\tau} \|_{\gamma + \frac{1}{2}} \| \mu_{\tau} \|_{\gamma} d\tau 
        \leq \frac{C}{R^{1/2}} \int \limits_{s}^{t} \| \mu_{\tau}\|_{2\gamma + 1} d\tau,
 \end{align*}
 where we have used that $\| \mu_{\tau} \|_{\gamma + \frac{1}{2}} \| \mu_{\tau} \|_{\gamma} \leq \| \mu_{\tau}\|_{\gamma + \frac{1}{2}}^2
 \leq \| \mu_{\tau} \|_{2\gamma + 1}$. Using the moment estimates from Corollary \ref{CORR:00} and Lemma \ref{LEMMA:02}
 we find a constant $C > 0$ such that $\sup_{N \geq 1}I_1  \leq C R^{-1/2}$.
 Hence it remains to prove that $I_2 \longrightarrow 0$ as $N \to \infty$ for any fixed $R > 0$.

 Fix $R > 0$ and recall that $\varphi$ is a smooth function on $\R_+$ satisfying $\1_{[0,1]} \leq \varphi \leq \1_{[0,2]}$.
 Define
 \begin{align*}
  H_{R,m}^1(\mu;x) &:= \varphi \left( \frac{\sup \limits_{\tau \in [s,t]} \langle v(\tau) \rangle^2}{m^2} \right) H_{R}(\mu;x),
  \\  H_{R,m}^2(\mu;x) &:= \left( 1 - \varphi \left( \frac{\sup\limits_{\tau \in [s,t]} \langle v(\tau) \rangle^2}{m^2}\right ) \right) H_R(\mu;x),
 \end{align*}
 and let $F_{R,m}^j$ be given by \eqref{EQ:24} with $H$ replaced by $H_{R,m}^j$, $j = 1,2$. Then we obtain $I_2 \leq J_1 + J_2$, where
 \begin{align*}
 J_j &= \left| \int \limits_{\mathcal{P}(D(\R_+;\R^{2d}))} |F_{R,m}^j(\mu)| d\pi^{(N)}(\mu) - \int \limits_{\mathcal{P}(D(\R_+;\R^{2d}))}|F_{R,m}^j(\mu)|d\pi^{(\infty)}(\mu) \right|, \ \ \ j = 1,2.
 \end{align*}
 For any $N,m \geq 1$ and $\mu \in \mathcal{P}(D(\R_+;\R^{2d}))$ we obtain for some constant $C$ independent of $N$
 \begin{align*}
  |F_{R,m}^2(\mu)| &\leq \int \limits_{D(\R_+;\R^{2d})} \1_{ \left\{ \sup \limits_{\tau \in [s,t]} \langle v(\tau) \rangle > m \right\} } |H_R(\mu;r,v)| \mu(dr,dv)
 \\ &\leq C \int \limits_{D(\R_+;\R^{2d})} \int \limits_{s}^t \1_{ \left\{ \sup \limits_{\tau \in [s,t]} \langle v(\tau) \rangle > m \right\} } \langle v(\tau) \rangle^{\gamma} d\tau \mu(dr,dv)
 \\ &\leq \frac{C}{m} \int \limits_{D(\R_+; \R^{2d})} \sup \limits_{\tau \in [s,t]} \langle v(\tau) \rangle^{1+ \gamma} \mu(dr,dv).
 \end{align*}
 The moment estimates from Corollary \ref{CORR:00} and a similar application of the Lemma of Fatou as in Lemma \ref{LEMMA:02} gives
 \[
  J_2 \leq \frac{C}{m} \int \limits_{\mathcal{P}(D(\R_+; \R^{2d}))} \int \limits_{D(\R_+;\R^{2d})} \sup \limits_{\tau \in [s,t]} \langle v(\tau) \rangle^{1+ \gamma} \mu(dr,dv) d( \pi^{(N)} + \pi^{(\infty)})(\mu)
  \leq \frac{C}{m} < \infty.
 \]
 Hence it suffices to show that $J_1 \longrightarrow 0$ as $N \to \infty$ for each fixed $R,m$.

 Note that $H_{R,m}^{1}$ is bounded and jointly continuous in $(\mu,r,v)$.
 Hence $F_{R,m}^1$ is continuous and bounded on $\mathcal{P}(D(\R_+;\R^{2d})) )$.
 Using the weak convergence $\pi^{(N)} \longrightarrow \pi^{(\infty)}$ as $N \to \infty$ we conclude that 
 also $J_1 \longrightarrow 0$ as $N \to \infty$, for each fixed $R,m$. 
 \end{proof}

\section{Uniqueness for bounded coefficients}
In this section we study uniqueness for the nonlinear martingale problem $(A,C_c^1(\R^{2d}), \mu_0)$ in the case where $\sigma$ is bounded, i.e. $\gamma = 0$.
The following is our main result in this case.
\begin{Theorem}
 Suppose that $\gamma = 0$. Then for each $\mu_0 \in \mathcal{P}(\R^{2d})$ there exists at most one solution to the 
 nonlinear martingale problem $(A,C_c^1(\R^{2d}), \mu_0)$.
 In particular, there exists at most one weak solution to the mean-field SDE \eqref{EQ:07}.
\end{Theorem}
The proof of this theorem is deduced from the following considerations. 
Given any solution $\mu$ to the nonlinear martingale problem $(A,C_c^1(\R^{2d}), \mu_0)$, then by taking expectations in \eqref{EQ:36} we see that 
its time-marginals $(\mu_t)_{t \geq 0}$ satisfy the nonlinear Fokker-Planck equation
\begin{align}\label{FPE}
 \langle g, \mu_t \rangle = \langle g, \mu_0 \rangle + \int \limits_{0}^{t}\langle A(\mu_s)g, \mu_s \rangle, \ \ t \geq 0, \ \ g \in C_c^1(\R^{2d}),
\end{align}
where $A(\mu_s)$ was defined in \eqref{EQ:09}. 
Then we prove uniqueness for \eqref{FPE}. Based on this uniqueness result, 
it suffices to study the corresponding linearized martingale problem where $(\mu_t)_{t \geq 0}$ appearing in the argument of $A(\mu_t)$
can be regarded as a fixed parameter. 
Uniqueness for the latter (time-inhomogeneous) martingale problem follows classically by uniqueness of its time-marginals.

\subsection{Uniqueness for the time-marginals}
In this section we study uniqueness and stability for the time-marginals, i.e. solutions to \eqref{FPE}.
More precisely, we prove an a priori bound for any two solutions to \eqref{FPE} with respect to the total variation distance
\[
 \| \mu - \nu \|_{\mathrm{TV}} = \sup \left\{  \langle g, \mu - \nu \rangle \ : \ g \in B(\R^{2d}), \ \ \| g \|_{\infty} \leq 1 \right\},
\]
where $B(\R^{2d})$ denotes the space of all bounded measurable functions on $\R^{2d}$.
The proof of such bound relies on a mild formulation of \eqref{FPE} described below.
\begin{Lemma}
 Let $(\mu_t)_{t \geq 0} \subset \mathcal{P}(\R^{2d})$ be given. Then $(\mu_t)_{t \geq 0}$ satisfies \eqref{FPE} if and only if
 \begin{align}\label{EQ:30}
  \langle g, \mu_t \rangle = \langle S(t)g, \mu_0\rangle + \int \limits_{0}^{t}\langle \mathcal{A}S(t-s)g, \mu_s \otimes \mu_s \rangle ds
 \end{align}
 holds for all $g \in C^1_c(\R^{2d})$, where $S(t-s)g(r,v) = g(r + (t-s)v,v)$ and
 \begin{align}\label{CS:01}
  (\mathcal{A}g)(r,v;q,w) = \psi(r-q)\sigma(v-w) \int \limits_{\R^{2d}}\left( g(r,w+u) - g(r,v)\right)a(u)du.
 \end{align}
 Moreover, \eqref{EQ:30} naturally extends to all $g \in B(\R^{2d})$.
\end{Lemma}
\begin{proof}
 Observe that for $g \in B(\R^{2d})$ we have
 \begin{align}\label{EQ:31}
  | \mathcal{A}S(t-s)g(r,v;q,w) | \leq 2 \| g \|_{\infty} \| \psi \|_{\infty} \| \sigma\|_{\infty}.
 \end{align}
 The assertion can be shown by dominated convergence and standard density arguments.
\end{proof}
The following is our main estimate for solutions to \eqref{FPE}.
\begin{Theorem}\label{UNIQUENESS:FPE:BOUNDED}
 Suppose that $\gamma = 0$ and let $(\mu_t)_{t \geq 0}$ and $(\nu_t)_{t \geq 0}$ be two solutions to \eqref{FPE}. Then 
 \[
  \| \mu_t - \nu_t \|_{\mathrm{TV}} \leq \| \mu_0 - \nu_0 \|_{\mathrm{TV}} \exp\left( 4 \| \psi \|_{\infty} \| \sigma \|_{\infty} t \right), \ \ t \geq 0.
 \] 
\end{Theorem}
\begin{proof}
 Let $g \in B(\R^{2d})$ be such that $\| g\|_{\infty} \leq 1$. Then, by \eqref{EQ:30},
 \begin{align*}
  \langle g, \mu_t - \nu_t\rangle &= \langle S(t)g, \mu_0 - \nu_0 \rangle + \int \limits_{0}^{t}\langle \mathcal{A}S(t-s)g, \mu_s \otimes \mu_s - \nu_s \otimes \nu_s \rangle ds
  \\ &= \langle S(t)g, \mu_0 - \nu_0 \rangle + \int \limits_{0}^{t} \langle \mathcal{A}S(t-s)g, \mu_s \otimes (\mu_s - \nu_s) \rangle ds
      \\ &\ \ \ + \int \limits_{0}^{t} \langle \mathcal{A}S(t-s)g, (\mu_s - \nu_s)\otimes \nu_s \rangle ds
  \\ &\leq \| \mu_0 - \nu_0 \|_{\mathrm{TV}} + 4 \| \psi \|_{\infty} \| \sigma \|_{\infty} \int \limits_{0}^{t} \| \mu_s - \nu_s \|_{\mathrm{TV}} ds,
 \end{align*}
 where we have used $\| S(t)g \|_{\infty} \leq 1$ and \eqref{EQ:31} to obtain 
 \begin{align*}
  \langle \mathcal{A}S(t-s)g, (\mu_s - \nu_s)\otimes \nu_s \rangle
  &\leq \sup \limits_{(q,w) \in \R^{2d}} \int \limits_{\R^{2d}} (\mathcal{A}S(t-s)g)(r,v;q,w) (\mu_s - \nu_s)(dr,dv)
  \\ &\leq 2 \| \psi \|_{\infty} \| \sigma \|_{\infty} \| \mu_s - \nu_s \|_{\mathrm{TV}}.
 \end{align*}
 Similarly one can show that
 \[
  \langle \mathcal{A}S(t-s)g, (\mu_s - \nu_s)\otimes \nu_s \rangle \leq 2 \| \psi \|_{\infty} \| \sigma \|_{\infty} \| \mu_s - \nu_s \|_{\mathrm{TV}}.
 \]
 Taking the supremum over all $g \in B(\R^{2d})$ with $\| g\|_{\infty} \leq 1$ and then applying the Gronwall lemma yields the assertion.
\end{proof}

\subsection{Uniqueness in law for the Vlasov-McKean equation}
Below we prove that the nonlinear martingale problem $(A,C_c^1(\R^{2d}), \mu_0)$ has at most one solution.
\begin{Proposition}
 Suppose that $\gamma = 0$ and let $\mu_0 \in \mathcal{P}(\R^{2d})$. 
 Then there exists at most one solution $\mu \in \mathcal{P}(D(\R_+;\R^{2d}))$ to the nonlinear martingale problem $(A,C_c^1(\R^{2d}), \mu_0)$.
\end{Proposition}
\begin{proof}
 Let $\mu$ and $\widetilde{\mu}$ be two solutions to the nonlinear martingale problem $(A,C_c^1(\R^{2d}), \mu_0)$.
 Their time-marginals $(\mu_t)_{t \geq 0}$ and $(\widetilde{\mu}_t)_{t \geq 0}$
 both solve \eqref{FPE} and hence coincide, i.e. $\mu_t = \widetilde{\mu}_t$, for all $t \geq 0$. Consequently
 \[
  g(x(t)) - g(x(0)) - \int \limits_{0}^{t}(A(\mu_s)g)(x(s))ds, \ \ t \geq 0
 \]
 is a martingale with respect to $\mu$ and $\widetilde{\mu}$, for any $g \in C_c^1(\R^{2d})$.
 From this we readily conclude that $\mu = \widetilde{\mu}$, provided there exists at most one solution $(\rho_t)_{t \geq 0}$ to the 
 time-inhomogeneous Fokker-Planck equation
 \[
  \langle g, \rho_t \rangle = \langle g, \rho_0 \rangle + \int \limits_{0}^{t}\langle A(\mu_s)g, \rho_s \rangle, \ \ t \geq 0, \ \ g \in C_c^1(\R^{2d}),
 \]
 apply e.g. \cite[p.184, Theorem 4.2]{EK86}.
 Uniqueness for $(\rho_t)_{t \geq 0}$ can be shown in exactly the same way as Theorem \ref{UNIQUENESS:FPE:BOUNDED}. 
\end{proof}

\section{Further uniqueness for unbounded coefficients}
In this section we provide some sufficient condition for uniqueness and stability of solutions to \eqref{FPE} in the case where $\gamma \in (0,2]$.
\begin{Definition}
 Let $\mu_0 \in \mathcal{P}(\R^{2d})$. A solution to \eqref{FPE} is a family $(\mu_t)_{t \geq 0} \subset \mathcal{P}(\R^{2d})$ satisfying
 \[
  \int \limits_{0}^{t} \int \limits_{\R^{2d}} \langle v \rangle^{\gamma} \mu_t(dr,dv) dt < \infty, \ \ \forall T> 0
 \]
 and \eqref{FPE} holds for all $g \in C_c^1(\R^{2d})$.
\end{Definition}
Note that the additional integrability condition imposed on $(\mu_t)_{t \geq 0}$ guarantees that $\langle A(\mu_s)g, \mu_s \rangle$ in \eqref{FPE} makes sense.
As before, it is not difficult to see that any solution to \eqref{FPE} still satisfies the mild formulation \eqref{EQ:30}.

\subsection{Estimate on the total variation distance}
For $\delta > 0$ let 
\[
 \mathcal{U}(\gamma,\delta) = \left\{ (\mu_t)_{t \geq 0} \ | \ \sup \limits_{t \in [0,T]} \mathcal{C}_{\gamma}(\delta, \mu_t) < \infty, \ \ \forall T > 0 \right \}
\]
where 
\begin{align}\label{EQ:32}
 \mathcal{C}_{\gamma}(\delta, \mu_t) := \int \limits_{\R^{2d}}e^{\delta \langle v \rangle^{\gamma}} \mu_t(dr,dv).
\end{align}
The following is the main result on uniqueness and stability for \eqref{FPE}.
\begin{Theorem}\label{TH:03}
 Fix $\delta > 0$. Then there exists a constant $C = C(\psi, \sigma,\delta) > 0$
 such that any two solutions $(\mu_t)_{t \geq 0}, (\nu_t)_{t \geq 0} \in \mathcal{U}(\gamma,\delta)$ to \eqref{FPE} satisfy
 \[
  \| \mu_t - \nu_t\|_{\mathrm{TV}} \leq \| \mu_0 - \nu_0 \|_{\mathrm{TV}} 
  + C \int \limits_{0}^{t}(\| \mu_s \|_{\gamma} + \| \nu_s \|_{\gamma})\mathcal{C}_{\gamma}(\delta, \mu_s + \nu_s) \| \mu_s - \nu_s\|_{\mathrm{TV}}
  (1 + |\ln( \| \mu_s - \nu_s\|_{\mathrm{TV}} )| ) ds.
 \]
 In particular, the following assertions hold
 \begin{enumerate}
  \item[(a)] There exists at most one solution to \eqref{FPE} in $\mathcal{U}(\gamma,\delta)$.
  \item[(b)] Let $\mu_0, \mu_0^{(n)} \in \mathcal{P}(\R^{2d})$ with
  \[
   \| \mu_0 - \mu_0^{(n)} \|_{\mathrm{TV}} \longrightarrow 0, \ \ n \to \infty
  \]
  and let $(\mu_t)_{t \geq 0}$ and $(\mu_t^{(n)})_{t \geq 0}$ be two solutions to \eqref{FPE} with initial condition $\mu_0$ and $\mu_0^{(n)}$, respectively.
  Suppose that there exists $\delta > 0$ such that
  \[
   \sup \limits_{n \geq 1}\sup \limits_{t \in [0,T]}\mathcal{C}_{\gamma}(\delta, \mu_t + \mu_t^{(n)}) < \infty, \ \ \forall T > 0.
  \]
  Then, for any $t \geq 0$,
  \[
   \| \mu_t - \mu_t^{(n)} \|_{\mathrm{TV}} \longrightarrow 0, \ \ n \to \infty.
  \]
 \end{enumerate}
\end{Theorem}
\begin{proof}
 Let $g \in B(\R^{2d})$ be such that $\| g\|_{\infty} \leq 1$. Using the mild formulation \eqref{EQ:30} we obtain
 \begin{align*}
  \langle g, \mu_t - \nu_t\rangle &= \langle S(t)g, \mu_0 - \nu_0 \rangle + \int \limits_{0}^{t}\langle \mathcal{A}S(t-s)g, \mu_s \otimes \mu_s - \nu_s \otimes \nu_s \rangle ds
  \\ &= \int \limits_{0}^{t}\langle \mathcal{A}S(t-s)g, \mu_s \otimes (\mu_s - \nu_s) \rangle ds
  + \int \limits_{0}^{t} \langle \mathcal{A}S(t-s)g, (\mu_s - \nu_s)\otimes \nu_s \rangle.
 \end{align*}
 Let $\varphi$ be a smooth function on $\R_+$ such that $\1_{[0,1]} \leq \varphi \leq \1_{[0,2]}$ and set 
 $\varphi_R(w) := \varphi\left( \frac{\langle w \rangle^2}{R^2}\right)$.
 Using the definition of $\mathcal{A}$ (see \eqref{CS:01}) and $(1 - \varphi_R(w)) \leq \1_{ \{ \langle w \rangle \geq R \} }$ we obtain
 \begin{align*}
  &\ \langle \mathcal{A}S(t-s)g, \mu_s \otimes (\mu_s - \nu_s) \rangle
  \\ &\leq \int \limits_{\R^{4d}}\varphi_R(w)(\mathcal{A}S(t-s)g)(r,v;q,w)d\mu_s(r,v)d(\mu_s-\nu_s)(q,w) 
  \\ &\ \ \ + \int \limits_{\R^{4d}}\1_{ \{ \langle w \rangle \geq R \} }|(\mathcal{A}S(t-s)g)(r,v;q,w)|d\mu_s(r,v)d(\mu_s+\nu_s)(q,w)
  \\ &\leq C \| \mu_s \|_{\gamma} R^{\gamma} \| \mu_s - \nu_s\|_{\mathrm{TV}}
         + C \| \mu_s \|_{\gamma} \int \limits_{\R^{4d}}\1_{ \{ \langle w \rangle \geq R \} } \langle w \rangle^{\gamma} d(\mu_s+\nu_s)(q,w).
 \end{align*}
 For the last term we use similar arguments to \cite{FM09} and \cite{FRS18b}.
 Namely, using $\langle w \rangle^{\gamma} \leq C e^{\frac{\delta}{2} \langle w \rangle^{\gamma}}$
 for some constant $C > 0$ large enough, we get
 \begin{align*}
  \int \limits_{\R^{4d}}\1_{ \{ \langle w \rangle \geq R \} } \langle w \rangle^{\gamma} d(\mu_s+\nu_s)(q,w)
  &\leq C \int \limits_{\R^{4d}} \1_{ \{ \langle w \rangle \geq R \} } e^{- \frac{\delta}{2}\langle w \rangle^{\gamma}} e^{\delta \langle w \rangle^{\gamma}}d(\mu_s + \nu_s)(q,w)
  \\ &\leq C e^{- \frac{\delta}{2}R^{\gamma}} \mathcal{C}_{\gamma}(\delta,\mu_s + \nu_s).
 \end{align*}
 Taking $R^{\gamma} = \frac{2}{\delta}| \ln( \| \mu_s - \nu_s \|_{\mathrm{TV}} )|$ we deduce
 \[
  \langle \mathcal{A}S(t-s)g, \mu_s \otimes (\mu_s - \nu_s) \rangle
  \leq C \mathcal{C}_{\gamma}(\delta,\mu_s + \nu_s) \| \mu_s \|_{\gamma} \| \mu_s - \nu_s\|_{\mathrm{TV}}(1 + |\ln( \| \mu_s - \nu_s \|_{\mathrm{TV}} )|).
 \]
 Proceeding in the same way we can show that
 \[
  \langle \mathcal{A}S(t-s)g, (\mu_s - \nu_s)\otimes \nu_s \rangle
  \leq C \mathcal{C}_{\gamma}(\delta,\mu_s + \nu_s) \| \nu_s \|_{\gamma} \| \mu_s - \nu_s\|_{\mathrm{TV}}(1 + |\ln( \| \mu_s - \nu_s \|_{\mathrm{TV}} )|),
 \]
 which proves the assertion after taking the supremum over all $g \in B(\R^{2d})$ with $\| g \|_{\infty} \leq 1$.
 Uniqueness and stability is a direct consequence of the a priori estimate we have shown, i.e.
 one may apply a generalization of the Gronwall inequality stated in the appendix.
\end{proof}

\subsection{Estimate on the Wasserstein distance}
In this part we prove estimates for solutions to \eqref{FPE} with respect to the Wasserstein distance
\[
 d(\mu, \nu) = \sup \limits_{\| g\|_0 \leq 1} \frac{|g(r,v) - g(\widetilde{r}, \widetilde{v})|}{|r - \widetilde{r}| + |v - \widetilde{v}|},
 \qquad \| g \|_0 := \sup \limits_{(r,v) \neq (\widetilde{r}, \widetilde{v})} \frac{|g(r,v) - g(\widetilde{r}, \widetilde{v})|}{|r - \widetilde{r}| + |v - \widetilde{v}|},
\]
where $\mu, \nu \in \mathcal{P}(\R^{2d})$ are supposed to have finite first moments.
Since particles are transported by the transport operator $v \cdot \nabla_r$, it is more natural to use the shifted Wasserstein distance
\[
 d_t(\mu,\nu) = d(S(-t)^*\mu, S(-t)^*\nu), \ \ t \geq 0,
\]
where $S(t)g(r,v) = g(r+vt,v)$ and $S(t)^*$ is the adjoint operator defined by the relation
\[
 \langle S(t)g, \mu \rangle = \langle g, S(t)^*\mu \rangle, \ \ g \in B(\R^{2d}), \ \ \mu \in \mathcal{P}(\R^{2d}).
\]
Below we will use another characterization of the shifted distance in terms of optimal couplings described as follows.

Introduce a one-paramter family of metrics on $\R^{2d}$
\[
 |(r,v) - (\widetilde{r},\widetilde{v})|_t := |(r-vt) - (\widetilde{r} - \widetilde{v}t)| + |v- \widetilde{v}|, \ \ t \geq 0
\]
and related to this metrics define the time-dependent Lipschitz norms 
\[
 \| g \|_{t} = \sup_{(r,v) \neq (\widetilde{r}, \widetilde{v})} \frac{|g(r,v) - g(\widetilde{r}, \widetilde{v})|}{|(r,v) - (\widetilde{r},\widetilde{v})|_t}.
\]
Note that this norms are all equivalent due to 
\begin{align*}
 \frac{1}{1+t}|(r,v) - (\widetilde{r},\widetilde{v})|_t \leq |(r,v) - (\widetilde{r},\widetilde{v})|_0 \leq (1+t)|(r,v) - (\widetilde{r},\widetilde{v})|_t.
\end{align*}
Given $\mu,\nu \in \mathcal{P}(\R^{2d})$, a coupling $H$ of $(\mu,\nu)$ is a probability measure on $\R^{4d}$ such that
its marginals are given by $\mu$ and $\nu$, respectively, i.e. for all $g_1,g_2 \in C_b(\R^{2d})$ one has
\[
 \int \limits_{\R^{4d}}\left( g_1(r,v) + g_2(\widetilde{r},\widetilde{v})\right)dH(r,v;\widetilde{r},\widetilde{v}) = \langle g_1, \mu \rangle + \langle g_2, \nu \rangle.
\]
Let $\mathcal{H}(\mu,\nu)$ the space of all such couplings.
The reader may consult \cite{V09} for additional details on couplings and related Wasserstein distance.
\begin{Proposition}
 Let $\mu,\nu \in \mathcal{P}\R^{2d})$ satisfy $\int_{\R^{2d}}(|r| + |v|)(\mu + \nu)(dr,dv) < \infty$ and fix $t \geq 0$. 
 Then there exists $H_t \in \mathcal{H}(\mu, \nu)$ such that
 \begin{align}\label{COUP:00}
  d_t(\mu, \nu) &= \sup\limits_{\| \psi \|_{0} \leq 1} \langle S(-t)\psi, \mu - \nu \rangle 
  = \sup\limits_{\| \psi \|_{t} \leq 1} \langle \psi, \mu - \nu \rangle
  = \int \limits_{\R^{4d}}|(r,v) - (\widetilde{r}, \widetilde{v})|_t dH_t(r,v;\widetilde{r},\widetilde{v}).
 \end{align}
\end{Proposition}
\begin{proof}
 The first equality follows from the definition of $S(t)^*$, the second equality from the definition of the norms $\| \cdot \|_t$
 while the third equality is a particular case of the Kantorovich-duality (see \cite{V09}).
\end{proof}
The following is our main coupling estimate for the Wasserstein distance $d_t$.
\begin{Proposition}
 Suppose that $\int_{\R^{2d}} |u| a(u)du < \infty$ and let $\mu_0, \nu_0 \in \mathcal{P}(\R^{2d})$ satisfy
 \[
  \int \limits_{\R^{2d}} (|r| + |v|) (\mu_0 + \nu_0)(dr,dv) < \infty.
 \]
 Let $(\mu_t)_{t \geq 0}$ and $(\nu_t)_{t \geq 0}$ be two solutions to \eqref{FPE} satisfying
 \[
  \int \limits_{0}^{T} \int \limits_{\R^{2d}} \left( |r| + |v|^{1 + \gamma} \right) (\mu_t + \nu_t)(dr,dv) < \infty, \ \ \forall T > 0.
 \]
 For $t \geq 0$, let $H_t \in \mathcal{H}(\mu_t,\nu_t)$ be such that
 \begin{align}\label{CS:02}
  d_t(\mu_t,\nu_t) = \int \limits_{\R^{4d}}|(r,v) - (\widetilde{r},\widetilde{v})|_t dH_t(r,v;\widetilde{r},\widetilde{v}).
 \end{align}
 Then there exists $C(T,a,\psi) > 0$ (independent of $\mu_t,\nu_t$) such that, for any $t \geq 0$,
 \begin{align*}
   d_t(\mu_t,\nu_t) \leq d_0(\mu_0,\nu_0) 
  + C(T,a,\psi) \int \limits_{0}^{t}\int \limits_{\R^{8d}} \Lambda(r,v,q,w; \widetilde{r},\widetilde{v},\widetilde{q},\widetilde{w}) dH_s^0 dH_s^1 ds
 \end{align*}
 where $dH_s^0 = dH_s(r,v;\widetilde{r},\widetilde{v})$, $dH_s^1 = dH_s(q,w;\widetilde{q},\widetilde{w})$ and
 \begin{align*}
  \Lambda(r,v,q,w; \widetilde{r},\widetilde{v},\widetilde{q},\widetilde{w}) 
 &= (  \langle v \rangle + \langle w \rangle + \langle \widetilde{v} \rangle + \langle \widetilde{w} \rangle )  | \sigma(v - w)\psi(r-q) - \sigma(\widetilde{v} - \widetilde{w})\psi(\widetilde{r} - \widetilde{q})|
 \\ &\ \ \  + \left( |(r,w) - (\widetilde{r}, \widetilde{w})|_s + |(r,v) - (\widetilde{r},\widetilde{v})|_s \right) \min\{ \sigma(v - w), \sigma(\widetilde{v} - \widetilde{w})\}
 \end{align*}
\end{Proposition}
\begin{proof}
 It is not difficult to see that both solutions still satisfy the mild formulation \eqref{EQ:30} for any $g$ with $\| g\|_0 \leq 1$.
 Hence we obtain
 \begin{align*}
  &\ \langle S(-t)g, \mu_t - \nu_t \rangle - \langle g, \mu_0 - \nu_0 \rangle
 \\ &= \int \limits_{0}^{t}\langle \mathcal{A}S(-s)g, \mu_s \otimes \mu_s - \nu_s \otimes \nu_s \rangle ds
 \\ &= \int \limits_{0}^{t} \int \limits_{\R^{8d}} \left[ (\mathcal{A}S(-s)g)(r,v;\widetilde{r},\widetilde{v}) - (\mathcal{A}S(-s)g)(q,w;\widetilde{q},\widetilde{w})\right] dH_s^0 dH_s^1ds 
 =: I.
 \end{align*}
 For simplicity of notation, let $\widetilde{\psi} = \psi(\widetilde{r}- \widetilde{q})$, $\widetilde{\sigma} = \sigma(\widetilde{v} - \widetilde{w})$
 and similarly $\psi = \psi(r-q)$ and $\sigma = \sigma(v-w)$.
 Using the definition of $dH_s^0dH_s^1$ together with $x = x\wedge y + (x-y)_+$, for $x,y \geq 0$, we obtain
 \begin{align*}
  I &= \int \limits_{0}^{t}\int \limits_{\R^{9d}}\bigg \{ \left( S(-s)g(r,w+u) - S(-s)g(r,v)\right)\psi \sigma
  \\ &\ \ \ \ \ \ \ \ \ \ - \left( S(-s)g(\widetilde{r}, \widetilde{w}+ u) - S(-s)g(\widetilde{r},\widetilde{v})\right)\widetilde{\psi}\widetilde{\sigma} \bigg \}a(u)dH_s^0 dH_s^1 ds
  \\ &\leq \int \limits_{0}^{t}\int \limits_{\R^{9d}}\bigg\{ S(-s)g(r,w+u) - S(-s)g(\widetilde{r}, \widetilde{w} + u) 
  \\ &\ \ \ \ \ \ \ \ \ \ +  S(-s)g(\widetilde{r}, \widetilde{v}) - S(-s)g(r,v)\bigg\} (\psi \sigma \wedge \widetilde{\psi} \widetilde{\sigma}) a(u)du dH_s^0 dH_s^1 ds
  \\ &\ \ \ + \int \limits_{0}^{t}\int \limits_{\R^{9d}}\left( S(-s)g(r,w+u) - S(-s)g(r,v)\right)\left( \psi \sigma - \widetilde{\psi}\widetilde{\sigma} \right)_+ a(u)du dH_s^0 dH_s^1 ds
  \\ &\ \ \ + \int \limits_{0}^{t}\int \limits_{\R^{9d}}\left( S(-s)g(\widetilde{r},\widetilde{w}+u) - S(-s)g(\widetilde{r},\widetilde{v})\right)\left( \widetilde{\psi} \widetilde{\sigma} - \psi \sigma \right)_+ a(u)du dH_s^0 dH_s^1 ds
  \\ &= J_1 + J_2 + J_3.
 \end{align*}
 Using $\| S(-s)g \|_s \leq 1$ we obtain
 \begin{align*}
  J_2 + J_3 &\leq \int \limits_{0}^{t}\int \limits_{\R^{9d}}\left\{ |(r,w+u) - (r,v)|_s + |(\widetilde{r}, \widetilde{w} + u) - (\widetilde{r},\widetilde{v})|_s\right\} \left| \psi \sigma - \widetilde{\psi}\widetilde{\sigma} \right| a(u)du dH_s^0 dH_s^1 ds
  \\ &\leq \int \limits_{0}^{t}(1+s)\int \limits_{\R^{9d}}\left( |w+u-v| + |\widetilde{w} + u - \widetilde{v}| \right)\left| \psi \sigma - \widetilde{\psi}\widetilde{\sigma} \right| a(u)du dH_s^0 dH_s^1 ds
  \\ &\leq C \int \limits_{0}^{t}\int \limits_{\R^{8d}}(  \langle v \rangle + \langle w \rangle + \langle \widetilde{v} \rangle + \langle \widetilde{w} \rangle )\left| \psi \sigma - \widetilde{\psi}\widetilde{\sigma} \right| dH_s^0 dH_s^1 ds
 \end{align*}
 where we have used $|w+u-v| + |\widetilde{w} + u - \widetilde{v}| \leq C \langle u \rangle (  \langle v \rangle + \langle w \rangle + \langle \widetilde{v} \rangle + \langle \widetilde{w} \rangle )$ 
 in the last inequality. Using again $\| S(-s)g\|_{s} \leq 1$ gives
 \begin{align*}
  S(-s)g(r,w+u) - S(-s)g(\widetilde{r}, \widetilde{w} + u) 
  &\leq |(r,w+u) - (\widetilde{r}, \widetilde{w} + u)|_s = |(r,w) - (\widetilde{r},\widetilde{w})|_s,
  \\ S(-s)g(\widetilde{r}, \widetilde{v}) - S(-s)g(r,v) &\leq |(r,v) - (\widetilde{r},\widetilde{v})|_s.
 \end{align*}
 Hence $J_1$ is estimated by
 \begin{align*}
  J_1 &\leq \int \limits_{0}^{t} \int \limits_{\R^{9d}}\left(  |(r,w) - (\widetilde{r},\widetilde{w})|_s +  |(r,v) - (\widetilde{r},\widetilde{v})|_s \right) (\psi \sigma \wedge \widetilde{\psi} \widetilde{\sigma}) a(u)du dH_s^0 dH_s^1 ds
 \\ &\leq  \| \psi \|_{\infty} \int \limits_0^t \int \limits_{\R^{8d}}\left(  |(r,w) - (\widetilde{r},\widetilde{w})|_s +  |(r,v) - (\widetilde{r},\widetilde{v})|_s \right) (\sigma \wedge \widetilde{\sigma}) dH_s^0 dH_s^1 ds
 \end{align*}
 which proves the assertion.
\end{proof}
The following gives the main estimate for this section.
\begin{Theorem}\label{TH:06}
 Suppose that $\int_{\R^{2d}} |u| a(u)du < \infty$ and assume that $\psi,\sigma$ are globally Lipschitz continuous.
 Then for each $\delta > 0$ and $T > 0$ there exists a constant $C = C(T, \delta,a, \psi,\sigma)$  such that for all $\mu_0, \nu_0 \in \mathcal{P}(\R^{2d})$ 
 any two solutions $(\mu_t)_{t \geq 0}, (\nu_t)_{t \geq 0}$ to \eqref{FPE} satisfying 
 \begin{align}\label{EQ:01}
  \mathcal{C}_{\gamma}(T, \mu+\nu, \delta) = \sup \limits_{t \in [0,T]}\int \limits_{\R^{2d}}\left( e^{\delta |v|^{1+\gamma}} + |r|^{1+\delta} \right)d(\mu_t + \nu_t)(r,v) < \infty
 \end{align}
 it holds that
 \begin{align*}
  d_t(\mu_t,\nu_t) &\leq d_0(\mu_0,\nu_0) +  C \mathcal{C}_{\gamma}(T, \mu+\nu, \delta)\int \limits_{0}^t d_s(\mu_s,\nu_s) (1 + |\ln(d_s(\mu_s,\nu_s))|) ds.
 \end{align*}
\end{Theorem}
\begin{proof}
 It is easily seen that the general coupling inequality is applicable in this case.
 Let us start with the first term in $\Lambda$.
 Using the elementary inequality
 \[
  c_{a,b}|x^{a+b} - y^{a+b}| \leq (x^a + y^a)|x^b - y^b| \leq C_{a,b}|x^{a+b} - y^{a+b}|, \ \ x,y \geq 0, \ \ a,b > 0
 \]
 we obtain
 \begin{align*}
  & \left| \sigma(v-w)\psi(r-q) - \sigma(\widetilde{v} - \widetilde{w}) \psi(\widetilde{r} - \widetilde{q}) \right|
\\ &\leq \sigma(v-w) \left| \psi(r-q) - \psi(\widetilde{r} - \widetilde{q}) \right| + \psi(\widetilde{r} - \widetilde{q})\left| \sigma(v-w) - \sigma(\widetilde{v} - \widetilde{w})\right|
 \\ &\leq C \left( \langle v \rangle^{\gamma} + \langle w \rangle^{\gamma} \right) \left( |r - \widetilde{r}| + |q - \widetilde{q}| \right)
 + C \left( |v - \widetilde{v} | + |w - \widetilde{w}| \right)
 \end{align*}
 and hence
 \begin{align*}
  & (  \langle v \rangle + \langle w \rangle + \langle \widetilde{v} \rangle + \langle \widetilde{w} \rangle ) | \sigma(v - w)\psi(r-q) - \sigma(\widetilde{v} - \widetilde{w})\psi(\widetilde{r} - \widetilde{q})| 
 \\ &\leq C \left(  \langle v \rangle^{1+\gamma} + \langle w \rangle^{1+\gamma} + \langle \widetilde{v} \rangle^{1 + \gamma} + \langle \widetilde{w} \rangle^{1+\gamma} \right)\left( |r - \widetilde{r}| + |q - \widetilde{q}| + |v - \widetilde{v} | + |w - \widetilde{w}| \right)
 \\ &\leq C \left(   \langle w \rangle^{1+\gamma} +  \langle \widetilde{w} \rangle^{1+\gamma} \right)\left( |r - \widetilde{r}| + |v - \widetilde{v} | \right)
 +  C \left(   \langle v \rangle^{1+\gamma} +  \langle \widetilde{v} \rangle^{1+\gamma} \right)\left( |q - \widetilde{q}| + |w - \widetilde{w} | \right)
\\ &\ \ \ +  C \left(   \langle v \rangle^{1+\gamma} +  \langle \widetilde{v} \rangle^{1+\gamma} \right)\left( |r - \widetilde{r}| + |v - \widetilde{v} | \right)
 +  C \left(   \langle w \rangle^{1+\gamma} +  \langle \widetilde{w} \rangle^{1+\gamma} \right)\left( |q - \widetilde{q}| + |w - \widetilde{w} | \right).
 \end{align*}
 Hence using that $H_s^0, H_s^1 \in \mathcal{H}(\mu_s,\nu_s)$ we obtain
\begin{align*}
  &\ \int \limits_{\R^{8d}}(  \langle v \rangle + \langle w \rangle + \langle \widetilde{v} \rangle + \langle \widetilde{w} \rangle ) | \sigma(v - w)\psi(r-q) - \sigma(\widetilde{v} - \widetilde{w})\psi(\widetilde{r} - \widetilde{q})| dH_s^0 dH_s^1
  \\ &\leq C \left( \| \mu_s \|_{1+\gamma} + \| \nu_s \|_{1+\gamma}\right) \int \limits_{\R^{4d}}\left( |r - \widetilde{r}| + |v - \widetilde{v}| \right) dH_s(r,v;\widetilde{r},\widetilde{v})
 \\ & \ \ \ + C \int \limits_{\R^{4d}}\left( \langle v \rangle^{1+\gamma} + \langle \widetilde{v} \rangle^{1+\gamma}\right)\left( |r - \widetilde{r}| + |v - \widetilde{v}| \right) dH_s(r,v;\widetilde{r},\widetilde{v})
 \\ &\leq C \mathcal{C}_{\gamma}(T, \mu+\nu, \delta) d_s(\mu_s,\nu_s) 
 \\ &\ \ \ +  C \mathcal{C}_{\gamma}(T, \mu+\nu, \delta) \int \limits_{\R^{4d}}\left( \langle v \rangle^{1+\gamma} + \langle \widetilde{v} \rangle^{1+\gamma}\right)  |(r,v) - (\widetilde{r}, \widetilde{v})|_s dH_s(r,v;\widetilde{r},\widetilde{v})
 \\ &\leq C \mathcal{C}_{\gamma}(T, \mu+\nu, \delta)d_s(\mu_s,\nu_s) (1 + |\ln(d_s(\mu_s,\nu_s))|) d_s(\mu_s,\nu_s),
 \end{align*}
 where we have used $ |r - \widetilde{r}| + |v - \widetilde{v}|  \leq (1+T) |(r,v) - (\widetilde{r}, \widetilde{v})|_s$, \eqref{CS:02} and similar arguments
 to the proof of Theorem \ref{TH:03} (see also \cite{FRS18b} and \cite{FM09}) to obtain
 \begin{align*}
  &\  \int \limits_{\R^{4d}}\left( \langle v \rangle^{1+\gamma} + \langle \widetilde{v}\rangle^{1+\gamma}\right) |(r,v) - (\widetilde{r},\widetilde{v})|_s  dH_s(r,v;\widetilde{r}, \widetilde{v})
 \\ &\leq C \mathcal{C}_{\gamma}(T, \mu+\nu, \delta)d_s(\mu_s,\nu_s)(1 + |\ln(d_s(\mu_s,\nu_s))|).
 \end{align*}
 For the second term in $\Lambda$ we use
 \begin{align*}
  |(r,w) - (\widetilde{r},\widetilde{w})|_s &\leq  |r-\widetilde{r}| + (1+s)|w - \widetilde{w}|
  \\ &\leq (1+T)|(r,v) - (\widetilde{r},\widetilde{v})|_s + (1+T)|(q,w) - (\widetilde{q},\widetilde{w})|_s
 \end{align*}
 to obtain
 \begin{align*}
  &\ \int \limits_{\R^{8d}}\left( |(r,w) - (\widetilde{r}, \widetilde{w})|_s + |(r,v) - (\widetilde{r},\widetilde{v})|_s \right) \min\{ \sigma(v - w), \sigma(\widetilde{v} - \widetilde{w})\} dH_s^0 dH_s^1
  \\ &\leq C \int \limits_{\R^{8d}}\left( |(r,v) - (\widetilde{r},\widetilde{v})|_s + |(q,w) - (\widetilde{q},\widetilde{w})|_s \right)\min\{ \sigma(v - w), \sigma(\widetilde{v} - \widetilde{w})\} dH_s^0 dH_s^1
  \\ &\leq C( \| \mu_s \|_{1+\gamma} + \| \nu_s \|_{1+\gamma}) \int \limits_{\R^{4d}} \langle v \rangle^{\gamma} |(r,v) - (\widetilde{r},\widetilde{v})|_s  dH_s(r,v;\widetilde{r}, \widetilde{v})
 \\ &\leq C \mathcal{C}_{\gamma}(T,\mu+\nu,\delta) d_s(\mu_s,\nu_s)(1 + |\ln(d_s(\mu_s,\nu_s))|).
 \end{align*}
 Applying the general coupling inequality and then above estimates proves the assertion.
\end{proof}
\begin{Remark}
 Using again Lemma \ref{UNIQ:LEMMA02} from the Appendix we may deduce from above estimate uniqueness and stability with respect to the Wasserstein metric.
\end{Remark}

\section{Appendix}

\subsection{Proof of Lemma \ref{LEMMA:03}}
 \textit{(a)} Applying the It\^{o} formula we obtain, for $g \in C_c^1(\R^{2d})$,
 \[
  g(R(t),V(t)) - g(R(0),V(0)) - \int \limits_{0}^{t}(A(\mu_s)g)(R(s),V(s))ds = M_g(t), \ \ t \geq 0
 \]
 where $(M_g(t))_{t \geq 0}$ is a local martingale. It suffices to show that $(M_g(t))_{t \geq 0}$ is, indeed, a martingale.
 For each $g \in C_c^1(\R^{2d})$ we find $C > 0$ with
 \begin{align*}
  |A(\mu_s)g(r,v)| \leq C \int \limits_{\R^{2d}}\langle w \rangle^{\gamma}d\mu_s(q,w) \langle v \rangle^{\gamma} = C \| \mu_s \|_{\gamma} \langle v \rangle^{\gamma}.
 \end{align*}
 This implies that
 \begin{align*}
  \E( \sup \limits_{s \in [0,t]} |M_g(t)| ) &\leq 2 \| g \|_{\infty} + \int \limits_{0}^{t} \E( |(A(\mu_s)g)(R(s),V(s))| ) ds
  \\ &\leq 2 \| g \|_{\infty} + C \int \limits_{0}^{t} \| \mu_s \|_{\gamma} \E( \langle V(s) \rangle^{\gamma}) ds 
  \\ &\leq 2 \| g \|_{\infty} + t \sup \limits_{s \in [0,t]} \| \mu_s \|_{\gamma}^2 < \infty,
 \end{align*}
 i.e. $(M_g(t))_{t \geq 0}$ is a martingale (see e.g. \cite[Theorem 46, p.36]{P05}).
 \\ \textit{(b)} Let $(q_t,w_t)$ be a measurable process defined on $([0,1], \mathcal{B}([0,1]), d\eta)$ such that $(q_t,w_t)$ has law $\mu_t$, for all $t \geq 0$,
 where $\mu_t$ denotes the time-marginal of $\mu$. Using \cite[Theorem A.1]{HK90} gives the 
 existence of a weak solution $(R,V)$ to \eqref{EQ:07} such that $(R,V)$ has law $\mu$.

\subsection{Proof of Lemma \ref{LEMMA:00}}
 By the mean-value Theorem we get
 \begin{align*}
  | v_j + u |^{2p} &= ( |v_j|^2 + |u|^2 + 2 v_j \cdot u )^p
  \\ &= (|v_j|^2 + |u|^2)^p + 2p (|v_j|^2 + |u|^2)^{p-1} (v_j \cdot u) 
  \\ &\ \ \ \ + 4 p(p-1)(v_j \cdot u)^2 \int \limits_{0}^{1}(1-t)\left( |v_j|^2 + |u|^2 + 2t (v_j \cdot u)\right)^{p-2}dt
 \end{align*}
 For the last integral we get by $2|v_j||u| \leq |v_j|^2 + |u|^2$ and $(a+b)^{q} \leq 2^{q}(a^q + b^q)$ for $q \geq 0$ and $a,b \geq 0$
 \begin{align*}
  &\ \left| 4 p(p-1)(v_j \cdot u)^2 \int \limits_{0}^{1}(1-t)\left( |v_j|^2 + |u|^2 + 2t (v_j \cdot u)\right)^{p-2}dt \right|
  \\ &\leq 4 p(p-1) (|v_j|^2 + |u|^2 + 2 |v_j| |u| )^{p-2}|v_j|^2 |u|^2 
  \\ &\leq p(p-1)2^{p} (|v_j|^2 + |u|^2)^{p-2} |v_j|^2 |u|^2 
  \\ &\leq p(p-1)2^{2p-2}\left( |v_j|^{2p-2}|u|^2 + |v_j|^2 |u|^{2p-4} \right)
  \\ &\leq p(p-1)2^{2p-2}\langle u \rangle^{2p} \left( |v_j|^{2p-2} + |v_j|^2 \right).
 \end{align*}
 Let $k_p = \lfloor \frac{p+1}{2} \rfloor$ where $\lfloor x \rfloor \in \Z$ is defined by $\lfloor x \rfloor \leq x < \lfloor x \rfloor +1$,
 set $\binom{p}{l} = \frac{p(p-1)\cdots (p-l-1)}{l!}$, for $l \geq 1$, and $\binom{p}{0} = 1$.
 Then we obtain by the fractional binomial expansion (see e.g. \cite[Lemma 3.1]{LM12})
 \begin{align*}
  (|v_j|^2 + |u|^2)^p &\leq |u|^{2p} + |v_j|^{2p} + \sum \limits_{l=1}^{k_p}\binom{p}{l}\left( |v_j|^{2l}|u|^{2p-2l} + |v_j|^{2p-2l}|u|^{2l}\right)
  \\ &\leq \langle u \rangle^{2p} + |v_j|^{2p} + \langle u \rangle^{2p}\sum \limits_{l=1}^{k_p}\binom{p}{l}\left( |v_j|^{2l} + |v_j|^{2p-2l}\right),
 \end{align*}
 where we have used $k_p \leq p$. Using the symmetry of $a$ we have $\int_{\R^d}(v_j \cdot u)a(u)du = 0$ and hence obtain
 \begin{align*}
  &\ \int \limits_{\R^d}\left( |v_j + u|^{2p} - |v_k|^{2p}\right)a(u)dvu
  \\ &\leq \int \limits_{\R^d}\left( (|v_j|^2 + |u|^2)^p - |v_k|^{2p}\right)a(u)du + p(p-1)2^{2p-2}\lambda_{2p} \left( |v_j|^{2p-2} + |v_j|^2 \right)
  \\ &\leq |v_j|^{2p} - |v_k|^{2p} + \lambda_{2p} + \lambda_{2p} \sum \limits_{l=1}^{k_p}\binom{p}{l}\left( |v_j|^{2l} + |v_j|^{2p-2l}\right)
  \\ &\ \ \ + p(p-1)2^{2p-2}\lambda_{2p} \left( |v_j|^{2p-2} + |v_j|^2 \right)
  \\ &\leq |v_j|^{2p} - |v_k|^{2p} + \lambda_{2p} + \lambda_{2p}\left( p(p-1)2^{2p-2} + \sum \limits_{l=1}^{k_p}\binom{p}{l} \right)\left( \langle v_j \rangle^{2k_p} + \langle v_j \rangle^{2p-2}\right)
  \\ &\leq |v_j|^{2p} - |v_k|^{2p} + \lambda_{2p} + \lambda_{2p}2^{3p-1} \left( \langle v_j \rangle^{2k_p} + \langle v_j \rangle^{2p-2}\right),
 \end{align*}
 where we have used $\sum_{l=1}^{k_p}\binom{p}{l} \leq 2^p \leq 2^{3p-2}$ and $p(p-1) \leq 2^{p}$ to obtain
 \begin{align*}
   p(p-1)2^{2p-2} + \sum \limits_{l=1}^{k_p}\binom{p}{l} \leq 2^{3p-2} + 2^p \leq 2^{3p-1}.
 \end{align*}
 By symmetry we obtain
 \begin{align}\label{EQ:37}
  \sum \limits_{k,j=1}^{N}\psi(r_k - r_j)\sigma(v_k - v_j)\left( |v_j|^{2p}  - |v_k|^{2p}\right) = 0
 \end{align}
 and hence
  \begin{align*}
  &\ \frac{1}{N^2}\sum \limits_{k,j=1}^{N}\psi(r_k-r_j)\sigma(v_k - v_j)\int \limits_{\R^{d}}\left( | v_j + u |^{2p} - | v_k|^{2p}\right)a(u)du
  \\ &\leq \lambda_{2p}2^{3p-1}\frac{C}{N^2}\sum \limits_{k,j=1}^{N}\left( \langle v_k \rangle^{\gamma} + \langle v_j \rangle^{\gamma}\right)\left( \langle v_j \rangle^{2k_p} + \langle v_j \rangle^{2p-2}\right)
  \\ &= \lambda_{2p}2^{3p-1}\frac{C}{N^2}\sum \limits_{k,j=1}^{N}\left( \langle v_k \rangle^{\gamma}\langle v_j \rangle^{2k_p} + \langle v_k \rangle^{\gamma}\langle v_j \rangle^{2p-2} + \langle v_j \rangle^{2k_p + \gamma} + \langle v_j \rangle^{2p-2 + \gamma} \right).
 \end{align*}
 Since $k_p \leq p-1$ we obtain from the Young inequality
 \begin{align*}
 \langle v_k \rangle^{\gamma}\langle v_j \rangle^{2k_p} \leq \langle v_k \rangle^{\gamma} \langle v_j \rangle^{2p-2} 
  \leq \frac{\gamma}{2p-2+\gamma}\langle v_k \rangle^{2p-2+\gamma} + \frac{2p-2}{2p-2+\gamma}\langle v_j \rangle^{2p-2+\gamma}.
 \end{align*}
 Next by $2k_p + \gamma \leq 2p-2+\gamma$ we obtain $\langle v_j \rangle^{2k_p + \gamma} \leq \langle v_j \rangle^{2p-2 + \gamma}$.
 Putting all estimates together we deduce the assertion.

\subsection{Some variants of the Gronwall lemma}
We need the following generalization of the Gronwall inequality (see \cite[Lemma 5.2.1, p. 89]{C95}) for a proof).
\begin{Lemma}\label{UNIQ:LEMMA02}
 Let $\rho$ be a nonnegative bounded function on $[0,T]$, $a \in [0, \infty)$ and $g$ be a strictly positive and non-decreasing function on $(0,\infty)$.
 Suppose that $\int_{0}^{1}\frac{dx}{g(x)} = \infty$ and 
 \[
  \rho(t) \leq a + \int \limits_{0}^{t}g(\rho(s))ds, \ \ t \in [0,T].
 \]
 Then
 \begin{enumerate}
  \item[(a)] If $a = 0$, then $\rho(t) = 0$ for all $t \in [0,T]$.
  \item[(b)] If $a > 0$, then $G(a) - G(\rho(t)) \leq t$ where $G(x) = \int_{x}^{1}\frac{dy}{g(y)}$.
 \end{enumerate}
\end{Lemma}
The following nonlinear generalization of the Gronwall lemma is a particular case of the Bihari-LaSalle inequality.
\begin{Lemma}\label{LASALLE}
  Let $f: \R_+ \longrightarrow \R_+$ be measurable and suppose that 
 \[
  f(t) \leq f(0) + K \int \limits_{0}^{t}f(s)^{1-\alpha} ds, \ \ t \geq 0
 \]
 for some $K \geq 0$ and $\alpha \in (0,1)$. Then for any $t \geq 0$
 \[
  f(t) \leq \left( f(0)^{\alpha} + \alpha K t \right)^{1 / \alpha} \leq 2^{1/ \alpha-1}f(0) + \frac{\left( 2 \alpha K  \right)^{1 / \alpha} }{2} t^{1 / \alpha}.
 \]
\end{Lemma}

\subsection{Some localization result}
Let $(E,\rho)$ be a complete, separable metric space.
Let $A \subset C_b(E) \times C(E)$ be a (multi-valued) operator such that there exists $1 \leq \psi \in C(E)$ with
\begin{align}\label{APPENDIX:00}
  |g| \leq K_f \psi, \ \ \forall (f,g) \in A
\end{align}
for some $K_f > 0$. Set $\mathcal{P}_{\psi} := \left\{ \mu \in \mathcal{P}(E) \ | \ \int_E \psi(x) d\mu(x) < \infty \right\}$.
Here and below $D(\R_+;E)$ denotes the Skorokhod space and $x$ the canonical process on $D(\R_+;E)$.
\begin{Definition}
 Let $\mu \in \mathcal{P}_{\psi}$. A solution to the martingale problem $(A, \mu)$ is a probability measure $\Pr_{\mu}$
 on $D(\R_+;E)$ such that
 \begin{enumerate}
  \item[(a)] $\Pr_{\mu}(x(0) \in A) = \mu(A)$ for all $A \in \mathcal{B}(E)$.
  \item[(b)] $\int_0^T \E_{\mu}( \psi(x(t)) ) dt < \infty$ for all $T > 0$.
  \item[(c)] For all $(f,g) \in A$
  \begin{align}\label{APPENDIX:01}
   f(x(t)) - f(x(0)) - \int \limits_{0}^{t}g(x(s))ds, \ \ t \geq 0
  \end{align}
  is a martingale w.r.t. $\Pr_{\mu}$.
 \end{enumerate}
\end{Definition}
When working with martingale problems the use of localization techniques such as \cite[Theorem 6.3, Corollary 6.4]{EK86} is essential.
However, the statements therein require that $A \subset C_b(E) \times B(E)$, i.e. $\psi = 1$.
Below we give one possible extension.
\begin{Theorem}\label{APPENDIX:TH00}
 Let $A \subset C_b(E) \times C(E)$ satisfy \eqref{APPENDIX:00} and $A_m \subset C_b(E) \times C(E)$ be such that
 $|g_m| \leq K_f \psi$ holds for $(f,g_m) \in A_m$ with a constant $K_f > 0$ independent of $m \geq 1$. 
 Suppose that there exists $\mu \in \mathcal{P}_{\psi}$ such that the following conditions hold:
 \begin{enumerate}
  \item[(i)] There exist open sets $(U_m)_{m \geq 1}$ with $\overline{U}_m \subset U_{m+1}$, $\bigcup_{m \geq 1}U_m = E$ and
  \[
   \left \{ (f, \1_{U_m}g) \ | \ (f,g) \in A_m \right\} = \left\{ (f, \1_{U_m}g) \ | \ (f,g) \in A \right\}, \ \ m \geq 1.
  \]
  Moreover $\1_{U_m}\psi$ is bounded for any $m \geq 1$.
  \item[(ii)] The martingale problem $(A_m, \rho)$ has for each $\rho \in \mathcal{P}(E)$ and each $m \geq 1$ a unique solution.
  \item[(iii)] We have 
  \[
   \lim \limits_{k \to \infty}\sup \limits_{m \geq k}\Pr_{\mu}^{m}\left(\tau_k \leq T\right) = 0, \ \ \forall T > 0
  \]
  where $\Pr_{\mu}^m$ is the unique solution to the martingale problem $(A_m, \mu)$ and
  \[
   \tau_k = \inf \{ t > 0 \ | \ x(t) \not \in U_k \text{ or } x(t-) \not \in U_k \}
  \]
  is a stopping time on $D(\R_+;E)$.
  \item[(iv)] There exists $p > 1$ such that for all $T > 0$ there exists $C(p,T) > 0$ satisfying
  \[
   \sup \limits_{m \geq 1}\sup \limits_{t \in [0,T]} \E_{\mu}^m\left( \psi(x(t))^p \right) \leq C(p,T),
  \]
  where $\E_{\mu}^m$ denotes the expectation w.r.t. $\Pr_{\mu}^m$.
 \end{enumerate}
 Then there exists a unique solution $\Pr_{\mu}$ to the martingale problem $(A,\mu)$. This solution satisfies
 \[
  \sup \limits_{t \in [0,T]} \E_{\mu}( \psi(x(t))^p) \leq C(p,T), \ \ T > 0.
 \]
\end{Theorem}
\begin{Remark}
 In several cases one may take $U_m = \{ x \in E \ | \ \psi(x) < m \}$. 
 In such a case condition (iii) is implied by
 \[
  \lim \limits_{k \to \infty} \sup \limits_{m \geq k} \Pr_{\mu}^m( \sup \limits_{t \in [0,T]} \psi(x(t)) \geq k ) = 0, \ \ \forall T > 0
 \]
 or the stronger condition
 \[
  \sup \limits_{m \geq 1} \E_{\mu}^m\left( \sup \limits_{t \in [0,T]} \psi(x(t)) \right) < \infty, \ \ \forall T > 0.
 \]
\end{Remark}
\begin{proof}
 \textit{Step 1.} Let $n \geq 1$, $0 \leq t_1,\dots, t_n \leq T$ and $H \in C_b(E^n)$. 
 Then (i), (ii) together with \cite[Chapter 4, Theorem 6.1]{EK86} yield
 \[
  \E_{\mu}^m\left( \1_{\tau_k > T} H(x(t_1),\dots, x(t_n)) \right) = \E_{\mu}^k\left( \1_{\tau_k > T}H(x(t_1),\dots, x(t_n)) \right), \ \ 1 \leq k \leq m.
 \]
 \textit{Step 2.} Let us prove that $\Pr_{\mu}^m \longrightarrow \Pr_{\mu}$ weakly in $\mathcal{P}(D(\R_+;E))$.
 
 Recall that the topology on $D(\R_+;E)$ may be obtained from the metric
 \[
  d(x,y) = \inf \limits_{\lambda \in \Lambda} \left( \gamma(\lambda) \vee \int \limits_{0}^{\infty}e^{-u}\sup \limits_{t \geq 0}q(x(t \wedge u), y(\lambda(t)\wedge u)) du \right)
 \]
 where $q := \rho \wedge 1$, $\gamma(\lambda) := \sup \limits_{0 \leq s < t}\left| \log\left(\frac{\lambda(t) - \lambda(s)}{t-s}\right)\right|$ and
 $\Lambda$ is the set of all strictly increasing, Lipschitz continuous functions $\lambda: [0, \infty) \longrightarrow [0,\infty)$ with $\gamma(\lambda) < \infty$
 (see \cite[p.117]{EK86}). For $H: D(\R_+; E) \longrightarrow \R$ let 
 \[
  \Vert H \Vert_{BL} = \Vert H \Vert_{\infty} + \sup \limits_{x \neq y} \frac{|H(x) - H(y)|}{d(x,y)}.
 \]
 Then it suffices to prove that $(\Pr_{\mu}^m)_{m \geq 1} \subset \mathcal{P}(D(\R_+;E))$ is a Cauchy sequence w.r.t. the metric 
 \[
  d_{BL}(P,Q) = \sup \limits_{\Vert H \Vert_{BL} \leq 1} \left| \int \limits_{D(\R_+;E)}H(x)dP(x) - \int \limits_{D(\R_+;E)}H(x)dQ(x)\right|.
 \]
 Take $H$ with $\Vert H \Vert_{BL} \leq 1$, $T > 0$, $1 \leq k < m$ and set $x^T := x(\cdot \wedge T)$, $H^T(x) := H(x^T)$. Then 
 \begin{align*}
  |\E_{\mu}^m(H) - \E_{\mu}^k(H)| &\leq |\E_{\mu}^m(H^T) - \E_{\mu}^m(H)| +  | \E_{\mu}^m(H^T) - \E_{\mu}^k(H^T)|  + |\E_{\mu}^k(H^T) - \E_{\mu}^k(H)|
  \\ &=: I_1 + I_2 + I_3.
 \end{align*}
 Then by Step 1 and $\1_{\tau_m > T} \geq \1_{\tau_k > T}$ we get 
 \begin{align*}
  I_2 &\leq \Vert H \Vert_{\infty}\left( \Pr_{\mu}^m(\tau_m \leq T) + \Pr_{\mu}^k(\tau_k \leq T) \right) + \Vert H \Vert_{\infty} \E_{\mu}^m\left(  \1_{\tau_m > T} - \1_{\tau_k > T} \right)
  \\ &= \Vert H \Vert_{\infty}\left( \Pr_{\mu}^m(\tau_m \leq T) + \Pr_{\mu}^k(\tau_k \leq T) \right) + \Vert H \Vert_{\infty} \left( \Pr_{\mu}^m(\tau_m > T) - \Pr_{\mu}^k(\tau_k > T)\right)
 \end{align*}
 which tends by (iii) clearly to zero. Moreover we have
 \begin{align*}
  I_1 &= \left| \E_{\mu}^m(H^T) - \E_{\mu}^m(H)\right| \leq \E_{\mu}^m(d(x^T,x))
  \\ &\leq \E_{\mu}^m\left( \int \limits_{0}^{\infty}e^{-u}\sup \limits_{t \geq 0}q(x(t \wedge u \wedge T), x(t \wedge u)) du \right) \leq e^{-T}
 \end{align*}
 and likewise $I_3 \leq e^{-T}$ which completes Step 2.
 \\ \textit{Step 3.} Let $\Pr_{\mu}$ be the limit of $\Pr_{\mu}^m$. 
 Using (iv), monotone convergence and the Lemma of Fatou one can show that 
 \[
  \sup \limits_{t \in [0,T]} \E_{\mu}\left( \psi(x(t))^p \right) \leq C(p,T), \ \ T > 0.
 \]
 \textit{Step 4.} Take $g \in C(E)$ such that there exists $K_g > 0$ with $|g| \leq K_g \psi$. We show that
 \[
  \lim \limits_{m \to \infty}\E_{\mu}^m\left( g(x(t)) \right) = \E_{\mu}\left( g(x(t)) \right), \ \ t \in D_{\mu}
 \]
 where $D_{\mu} = \{ t \geq 0 \ | \ \Pr_{\mu}( x(t) = x(t-)) = 1 \}$. Note that $D_{\mu}^c$ is at most countable.
 
 Let $h_k \in C_b(E)$ be such that $\1_{U_k} \leq h_k \leq \1_{U_{k+1}}$, $k \geq 1$. Then for $k < m$
 \begin{align*}
  |\E_{\mu}^m(g(x(t)) - \E_{\mu}(g(x(t))| &\leq | \E_{\mu}^m(h_k(x(t))g(x(t))) - \E_{\mu}(h_k(x(t))g(x(t)))| 
  \\ &\ \ \ + |\E_{\mu}^m( (1 - h_k(x(t)))g(x(t)))| + |\E_{\mu}( (1 - h_k(x(t)))g(x(t)))|
  \\ &= I_1 + I_2 + I_3.
 \end{align*}
 It suffices to show that
 \begin{align*}
  \lim \limits_{m \to \infty} I_1 = 0, \ \ \forall k \geq 1
  \\ \lim \limits_{k \to \infty}\sup \limits_{m \geq k}(I_2 + I_3) = 0.
 \end{align*}
 Concerning $I_1$ the assertion follows by Step 2 and since $x \longmapsto h_k(x(t))g(x(t))$ is bounded and $\Pr_{\mu}$-a.s. continuous on $D(\R_+;E)$ for any $k \geq 1$.
 For the second property we use $\1_{U_k} ( 1 - h_k) = 0$ so that
 \begin{align*}
  I_2 + I_3 &= |\E_{\mu}^m( \1_{\tau_k \leq t} (1 - h_k(x(t)))g(x(t)))| + |\E_{\mu}( \1_{\tau_k \leq t}(1 - h_k(x(t)))g(x(t)))|
  \\ &\leq K_g \E_{\mu}^m( \1_{\tau_k \leq t} \psi(x(t))) + K_g \E_{\mu}( \1_{\tau_k \leq t} \psi(x(t)))
  \\ &\leq K_g (\Pr_{\mu}^m(\tau_k \leq t))^{1 - \frac{1}{p}}\E_{\mu}^m( \psi(x(t))^p)^{\frac{1}{p}} 
     + K_g (\Pr_{\mu}(\tau_k \leq t))^{1 - \frac{1}{p}}\E_{\mu}( \psi(x(t))^p)^{\frac{1}{p}}.
 \end{align*}
 For the first term we can use (iii) and (iv); for the second term this follows from $\Pr_{\mu} \in \mathcal{P}(D(\R_+;E))$.
 \\ \textit{Step 5.} $\Pr_{\mu}$ is a solution for the martingale problem for $(A,\mu)$.
 
 Fix $n \geq 1$, $0 \leq t_1,\dots, t_n\leq s < t$ in $D_{\mu}$, $h_1, \dots, h_n \in C_b(E)$, $(f,g) \in A$ and set 
 \begin{align}\label{APPENDIX:02}
  H := \left( f(x(t)) - f(x(s)) - \int \limits_{s}^{t}g(x(s))ds\right) \prod \limits_{k=1}^{n}h_k(x(t_k)).
 \end{align}
 We have to show that $\E_{\mu}(H) = 0$. First using Steps 3 and 4 together with (iv) and dominated convergence we easily deduce
 \begin{align*}
  \E_{\mu}(H) = \lim \limits_{m \to \infty} \E_{\mu}^m(H) = \lim \limits_{m \to \infty}\E_{\mu}^m( \1_{\tau_m \leq T} H) + \lim \limits_{m \to \infty}\E_{\mu}^m(\1_{\tau_m > T} H)
 \end{align*}
 where $t < T$. We can find a constant $C > 0$ such that
 \begin{align}\label{APPENDIX:03}
  | \E_{\mu}^m( \1_{\tau _m \leq T} H) | \leq C \Pr_{\mu}^m(\tau_m \leq T) + C \Pr_{\mu}^m(\tau_m \leq T)^{1 - \frac{1}{p}} \sup \limits_{t \in [0,T]}\E_{\mu}^m(\psi(x(t))^p)^{\frac{1}{p}}
 \end{align}
 and the right-hand side tends to zero as $m \to \infty$. Since $(f,g) \in A$ we can find by (i) $g_m \in C_b(E)$ such that $(f,g_m) \in A_m$
 and $\1_{U_m}g = \1_{U_m}g_m$. Let $H_m$ be given by \eqref{APPENDIX:02} with $g$ replaced by $g_m$. 
 Then, since $\Pr_{\mu}^m$ is a solution to the martingale problem $(A_m, \mu)$, it follows $\E_{\mu}^m(H_m) = 0$ and hence
 \[
  \E_{\mu}^m(\1_{\tau_m > T}H) = \E_{\mu}^m(\1_{\tau_m > T}H_m) = - \E_{\mu}^m(\1_{\tau_m \leq T} H_m).
 \]
 Since $|g_m| \leq C \psi$ for some $C > 0$ independent of $m$, the latter expression can be estimated in the same way as \eqref{APPENDIX:03}.
 \\ \textit{Step 6.} It remains to show that there exists only one solution to the martingale problem $(A,\mu)$.
 Let $\Pr_{\mu}' \in \mathcal{P}(D(\R_+;E))$ be any solution to the martingale problem $(A, \mu)$.
 Let $n \geq 1$, $0 \leq t_1,\dots, t_n \leq T$ and $H \in C_b(E^n)$.
 Then (ii) implies that
 \[
  \E_{\mu}^m(\1_{\tau_m > T}H(x(t_1), \dots, x(t_n)) ) = \E_{\mu}'( \1_{\tau_m > T} H(x(t_1), \dots, x(t_n)) ).
 \]
 The assertion now follows from the identity
 \begin{align*}
  &\ \E_{\mu}^m(H(x(t_1),\dots, x(t_n))) - \E_{\mu}'(H(x(t_1),\dots, x(t_n))) 
  \\ &= \E_{\mu}^m(\1_{\tau_m \leq T}H(x(t_1),\dots, x(t_n))) - \E_{\mu}'(\1_{\tau_m \leq T}H(x(t_1),\dots, x(t_n))) 
 \end{align*}
 after taking the limit $m \to \infty$.
\end{proof}

\newpage

\begin{footnotesize}

\bibliographystyle{alpha}
\bibliography{Bibliography}

\newcommand{\etalchar}[1]{$^{#1}$}
\begin{thebibliography}{HJN{\etalchar{+}}17}

\bibitem[AH10]{AH10}
Shin~Mi Ahn and Seung-Yeal Ha.
\newblock Stochastic flocking dynamics of the {C}ucker-{S}male model with
  multiplicative white noises.
\newblock {\em J. Math. Phys.}, 51(10):103301, 17, 2010.

\bibitem[Che95]{C95}
Jean-Yves Chemin.
\newblock Fluides parfaits incompressibles.
\newblock {\em Ast\'erisque}, (230):177, 1995.

\bibitem[CHZ18]{CHZ18}
Chiun-Chuan Chen, Seung-Yeal Ha, and Xiongtao Zhang.
\newblock The global well-posedness of the kinetic {C}ucker-{S}male flocking
  model with chemotactic movements.
\newblock {\em Commun. Pure Appl. Anal.}, 17(2):505--538, 2018.

\bibitem[CS07a]{CS07b}
Felipe Cucker and Steve Smale.
\newblock Emergent behavior in flocks.
\newblock {\em IEEE Trans. Automat. Control}, 52(5):852--862, 2007.

\bibitem[CS07b]{CS07a}
Felipe Cucker and Steve Smale.
\newblock On the mathematics of emergence.
\newblock {\em Jpn. J. Math.}, 2(1):197--227, 2007.

\bibitem[EK86]{EK86}
Stewart Ethier and Thomas~G. Kurtz.
\newblock {\em Markov processes}.
\newblock Wiley Series in Probability and Mathematical Statistics: Probability
  and Mathematical Statistics. John Wiles \& Sons, Inc., New York, 1986.
\newblock Characterization and convergence.

\bibitem[FM09]{FM09}
Nicolas Fournier and Cl\'ement Mouhot.
\newblock On the well-posedness of the spatially homogeneous {B}oltzmann
  equation with a moderate angular singularity.
\newblock {\em Comm. Math. Phys.}, 289(3):803--824, 2009.

\bibitem[FRS18a]{FRS18a}
Martin Friesen, Barbara R\"udiger, and Padmanabhan Sundar.
\newblock The {E}nskog process for hard and soft potentials.
\newblock {\em arXiv:1805.10141 [math-ph]}, 2018.

\bibitem[FRS18b]{FRS18b}
Martin Friesen, Barbara R\"udiger, and Padmanabhan Sundar.
\newblock On uniqueness for the {E}nskog process for hard and soft potentials.
\newblock {\em (to appear)}, 2018.

\bibitem[Gra92]{G92}
Carl Graham.
\newblock Mc{K}ean-{V}lasov {I}t\^o-{S}korohod equations, and nonlinear
  diffusions with discrete jump sets.
\newblock {\em Stochastic Process. Appl.}, 40(1):69--82, 1992.

\bibitem[HJN{\etalchar{+}}17]{HJNXZ17}
Seung-Yeal Ha, Jiin Jeong, Se~Eun Noh, Qinghua Xiao, and Xiongtao Zhang.
\newblock Emergent dynamics of {C}ucker-{S}male flocking particles in a random
  environment.
\newblock {\em J. Differential Equations}, 262(3):2554--2591, 2017.

\bibitem[HK90]{HK90}
Joseph Horowitz and Rajeeva~L. Karandikar.
\newblock Martingale problems associated with the {B}oltzmann equation.
\newblock In {\em Seminar on {S}tochastic {P}rocesses, 1989 ({S}an {D}iego,
  {CA}, 1989)}, volume~18 of {\em Progr. Probab.}, pages 75--122. Birkh\"auser
  Boston, Boston, MA, 1990.

\bibitem[HL09a]{HL09a}
Seung-Yeal Ha and Doron Levy.
\newblock Particle, kinetic and fluid models for phototaxis.
\newblock {\em Discrete Contin. Dyn. Syst. Ser. B}, 12(1):77--108, 2009.

\bibitem[HL09b]{HL09}
Seung-Yeal Ha and Jian-Guo Liu.
\newblock A simple proof of the {C}ucker-{S}male flocking dynamics and
  mean-field limit.
\newblock {\em Commun. Math. Sci.}, 7(2):297--325, 2009.

\bibitem[HT08]{HT08}
Seung-Yeal Ha and Eitan Tadmor.
\newblock From particle to kinetic and hydrodynamic descriptions of flocking.
\newblock {\em Kinet. Relat. Models}, 1(3):415--435, 2008.

\bibitem[JS03]{JS03}
Jean Jacod and Albert~N. Shiryaev.
\newblock {\em Limit theorems for stochastic processes}, volume 288 of {\em
  Grundlehren der Mathematischen Wissenschaften [Fundamental Principles of
  Mathematical Sciences]}.
\newblock Springer-Verlag, Berlin, second edition, 2003.

\bibitem[Kur11]{KURTZ10}
Thomas~G. Kurtz.
\newblock Equivalence of stochastic equations and martingale problems.
\newblock In {\em Stochastic analysis 2010}, pages 113--130. Springer
  Heidelberg, 2011.

\bibitem[LM12]{LM12}
Xuguang Lu and Cl\'ement Mouhot.
\newblock On measure solutions of the {B}oltzmann equation, part {I}: moment
  production and stability estimates.
\newblock {\em J. Differential Equations}, 252(4):3305--3363, 2012.

\bibitem[Pro05]{P05}
Philip~E. Protter.
\newblock {\em Stochastic integration and differential equations}, volume~21 of
  {\em Stochastic Modelling and Applied Probability}.
\newblock Springer-Verlag, Berlin, 2005.
\newblock Second edition. Version 2.1, Corrected third printing.

\bibitem[PRT15]{PRT15}
Benedetto Piccoli, Francesco Rossi, and Emmanuel Tr\'elat.
\newblock Control to flocking of the kinetic {C}ucker-{S}male model.
\newblock {\em SIAM J. Math. Anal.}, 47(6):4685--4719, 2015.

\bibitem[She08]{S08}
Jackie Shen.
\newblock Cucker-{S}male flocking under hierarchical leadership.
\newblock {\em SIAM J. Appl. Math.}, 68(3):694--719, 2007/08.

\bibitem[Szn91]{S91}
Alain-Sol Sznitman.
\newblock Topics in propagation of chaos.
\newblock In {\em \'Ecole d'\'Et\'e de {P}robabilit\'es de {S}aint-{F}lour
  {XIX}---1989}, volume 1464 of {\em Lecture Notes in Math.}, pages 165--251.
  Springer, Berlin, 1991.

\bibitem[Vil09]{V09}
C\'edric Villani.
\newblock {\em Optimal transport}, volume 338 of {\em Grundlehren der
  Mathematischen Wissenschaften [Fundamental Principles of Mathematical
  Sciences]}.
\newblock Springer-Verlag, Berlin, 2009.
\newblock Old and new.

\end{thebibliography}

\end{footnotesize}

\end{document}